\documentclass[12pt]{article}
\usepackage{setspace}
\usepackage[margin=1in]{geometry}
\usepackage{microtype}
\usepackage{amsmath,nccmath}
\usepackage{amssymb}
\usepackage{amsthm}
\usepackage{graphicx,multirow}
\usepackage{algorithm}
\usepackage{algorithmic}
\usepackage{caption}
\usepackage{float}
\usepackage{subfigure}
\usepackage{multicol,multirow,float}
\usepackage{bbm}
\usepackage{url}
\usepackage{mathtools}
\usepackage{tikz,pgfplots}
\usepackage{booktabs} 
\usepackage{adjustbox}
\usetikzlibrary{matrix}
\usepgfplotslibrary{groupplots}
\usepackage{authblk}
\usepackage[numbers,sort]{natbib}
\usepackage[pdfborder={0 0 0},breaklinks=true]{hyperref}
\usepackage{breakcites}

\newtheorem{theorem}{Theorem}
\newtheorem{lemma}[theorem]{Lemma}
\newtheorem{corollary}[theorem]{Corollary}
\newtheorem{property}[theorem]{Property}

\newcommand{\argmin}{\operatornamewithlimits{argmin}}

\newcommand{\ind}{\mathbbm{1}}

\def\ZZ{\mathbb{Z}}
\def\QQ{\mathbb{Q}}

\def\RR{\mathbb{R}}
\def\NN{\mathbb{N}}
\def\EE{\mathbb{E}}
\def\PP{\mathbb{P}}

\def\B{{\cal B}}

\def\h{\tilde{h}}

\def\LL{{\cal L}}

\def\F{{\cal F}}

\def\R{{\cal R}}

\def\st{{\rm s.t.}}

\DeclareMathOperator*{\sgn}{sgn}
\DeclareMathOperator*{\dist}{dist}

\newcommand{\z}{\textbf{0}}
\DeclarePairedDelimiter{\norm}{\lVert}{\rVert}

\newtheorem{prop}[theorem]{Property}
\newtheorem{ex}[theorem]{Example}

\newtheorem{innercustomproperty}{Property}  
\newenvironment{customproperty}[1]{\renewcommand\theinnercustomproperty{#1}\innercustomproperty}{\endinnercustomproperty}

\newtheorem{innercustomcorollary}{Corollary}  
\newenvironment{customcorollary}[1]{\renewcommand\theinnercustomcorollary{#1}\innercustomcorollary}{\endinnercustomcorollary}

\newtheorem{innercustomlemma}{Lemma}  
\newenvironment{customlemma}[1]{\renewcommand\theinnercustomlemma{#1}\innercustomlemma}{\endinnercustomlemma}

\begin{document}

\onecolumn
\title{Perturbed Iterate SGD for\\Lipschitz Continuous Loss Functions}

\author[1]{Michael R. Metel\thanks{michael.metel@huawei.com}}
\author[2]{Akiko Takeda\thanks{takeda@mist.i.u-tokyo.ac.jp}}

\affil[1]{Huawei Noah's Ark Lab, Montr\'eal, Qu\'ebec, Canada}
\affil[2]{Department of Creative Informatics, Graduate School of Information Science and Technology,
	The University of Tokyo, Tokyo, Japan; RIKEN Center for Advanced Intelligence Project, Tokyo, 
	Japan}

\maketitle

\begin{abstract}
This paper presents an extension of stochastic gradient descent for the minimization of Lipschitz continuous loss functions. Our motivation is for use in non-smooth non-convex stochastic optimization problems, which are frequently encountered 
in applications such as machine learning. Using the Clarke 
$\epsilon$-subdifferential, we prove the non-asymptotic convergence to an approximate stationary point in expectation for the proposed method. 
From this result, a method with non-asymptotic convergence with high probability, as well as a method with asymptotic convergence to a Clarke stationary point almost surely are developed. Our results hold under the assumption that the stochastic 
loss function is a Carath\'eodory function which is almost everywhere Lipschitz continuous in the decision variables. To the best of our knowledge this is the first non-asymptotic convergence analysis under these minimal assumptions.\\

\end{abstract}

\section{Introduction}
\label{int}
The focus of this work is on unconstrained minimization problems of the form  
\begin{alignat}{6}\label{eq:1}
&\min\limits_{w\in\RR^d}&&\text{ }f(w):=\EE_{\xi}[F(w,\xi)],
\end{alignat}
where $\xi\in\RR^p$ is a random vector from a probability space $(\Omega,\F,P)$. We assume that 
$F(w,\xi)$ is a Carath\'eodory function (see Section \ref{pre}) which is Lipschitz continuous in $w$ 
for $\xi$ almost 
everywhere, implying that $f(w)$ is a Lipschitz continuous function. Given observed samples $\widehat{\xi}_j$ of $\xi$ for 
$j=1,2,...,n$, our work is also 
applicable when considering the expectation with respect to the samples' associated empirical 
probability distribution, 
\begin{alignat}{6}\label{eq:0}
&f(w):=\frac{1}{n}\sum_{j=1}^n F(w,\widehat{\xi}_j).
\end{alignat}

We do not assume 
the loss function to be differentiable nor convex. This class of functions is 
quite general and enables our work to be applicable for a wide range of loss functions used in 
practice. The lack of a smoothness assumption allows for functions used in deep learning models, such 
as the ReLU activation function, and functions used in sparse learning models, such as $L^1$-norm 
regularization\footnote{Any deterministic Lipschitz continuous function can be added to 
$f(w)$ without changing our analysis.}. Further removing the requirement of convexity enables us 
to consider bounded loss functions, which are known to be more robust to outliers \cite{yang2010}, 
including all bounded functions with Lipschitz continuous gradients, see Property \ref{LCGprop} in the 
appendix.\\

We propose a new first-order stochastic method with 
non-asymptotic convergence bounds for finding an approximate stationary point of problem 
\eqref{eq:1} in expectation in terms of the Clarke $\epsilon$-subdifferential \cite{goldstein1977}. 
In this paper $\partial f(w)$ will denote the Clarke subdifferential (see equation \eqref{clsub} in 
Section \ref{pre}). A 
necessary condition for a point $\bar{w}$ to be a minimizer of $f(w)$ is for it to be a stationary point, i.e. $0\in\partial f(\bar{w})$. It has been recently proven though that for $\epsilon\in 
[0,1)$, there is no finite time algorithm which can guarantee an $\epsilon$-stationary point, 
$\dist(0,\partial f(w))\leq \epsilon$, for a 
directionally differentiable, bounded from below Lipschitz continuous function $f(w)$ using a family of 1-dimensional functions, see \cite[Theorem 5]{zhang2020}. 
Another difficulty is that without further assumptions placed on $f(w)$, $\partial f(w)\subset \EE_{\xi}[\partial F(w,\xi)]$ 
\cite[Theorem 2.7.2]{clarke1990}, 
so an unbiased estimate of a subgradient of $f(w)$ cannot be guaranteed by sampling stochastic 
subgradients of $\partial F(w,\xi)$. Motivated by the Gradient Sampling algorithm \cite{burke2018}, 
these issues can be resolved by using a step direction computed by sampling the gradient of the stochastic function $F(w,\xi)$ at random points near each iterate. Our work 
is also motivated by papers such as 
\cite{lakshmanan2008,yousefian2012,nesterov2017}, where the approximate function, 
$f_{\sigma}(w)=\EE_{z}[f(w+\sigma z)]$, is shown to have a Lipschitz 
continuous gradient when $z$ follows a Normal distribution or a uniform 
distribution in a Euclidean ball. It has been shown in \cite[Section 4.2]{kornowski2021} that this is an essentially optimal method of smoothing a non-smooth function.\\ 

The next section contains a discussion about related works, where we highlight previous papers proving 
asymptotic convergence for algorithms minimizing Lipschitz continuous functions, as well as 
two papers which have proven non-asymptotic convergence results, with comparisons to this 
paper. Section \ref{pre} contains our assumptions and notation, the 
convergence criterion, as well as some necessary properties concerning  Carath\'eodory 
functions which are Lipschitz continuous almost everywhere. Section \ref{ssd} presents the algorithm 
Perturbed Iterate SGD (PISGD), with the main convergence result contained in Theorem \ref{SGDF}, and accompanying computational complexity results for PISGD given in Corollaries \ref{ssdcomp}, \ref{ssdcomp2}, and \ref{ssdcomp3}. In particular, Corollary \ref{ssdcomp3} gives the computational complexity to find an approximate stationary point with high probability. A method to asymptotically converge to a Clarke stationary point almost surely is given in Corollary \ref{asymptotic}, and in Subsection \ref{settings} there is a discussion of applying PISGD for particular cases of $f(w)$.
In Section \ref{num}, the developed convergence theory is applied to train a Lipschitz continuous feedforward neural network. The conclusion is given in Section \ref{con}, and the appendix contains proofs and some auxiliary results.

\section{Related works} 
\label{rl}

\subsection{Asymptotic Convergence Analysis}

The Gradient Sampling algorithm \cite{burke2018} is proven to have asymptotic convergence to a 
stationary point for locally Lipschitz continuous functions which are continuously differentiable on an open set with full measure. In each iteration, the Gradient 
Sampling algorithm computes the minimum-norm vector $g$ in 
the convex hull of gradients sampled in a Euclidean ball near the current iterate, and uses a 
line-search method to determine the step size in the direction of $-g$, while also needing to ensure 
that the function is differentiable at the next iterate. The Gradient Sampling algorithm 
was inspired by the earlier analysis \cite{burke2002} of approximating the Clarke subdifferential by 
the convex hull of gradients sampled in a Euclidean ball.\\ 

In \cite{davis2018}, the stochastic subgradient 
algorithm is studied, where it was proven that almost surely, every limit point is a Clarke 
stationary point for all locally Lipschitz continuous functions whose graphs 
are Whitney stratifiable, assuming iterates are bounded almost surely. This class of functions 
includes a wide range of applications, including standard architectures in deep learning.\\ 

Whereas the previous work's results require vanishing step sizes, in \cite{bianchi2020} the convergence of constant step SGD is studied, where it is shown that SGD converges in probability to a Clarke stationary point when the constant step size approaches zero. Their problem setup is also similar to our own as they consider a Carath\'eodory function which is locally Lipschitz continuous in $w$ for each $\xi$.\\    

Another 
interesting paper which studies non-convex, non-differentiable minimization is 
\cite{bertsekas1975}, which considers objective functions whose 
non-differentiability stems from 
{\it simple kinks} of the form $\max\{0,f_i(w)\}$ for a family of functions $\{f_i(w)\}$. A smoothing 
method is developed, inspired by penalty and multiplier methods for 
constrained optimization problems, and convergence to an optimal solution is established by minimizing 
increasingly accurate approximations of the original problem with a boundedness assumption on the 
partial derivatives of the objective.

\subsection{Non-Asymptotic Convergence Analysis}

In \cite{nesterov2017}, non-asymptotic 
convergence bounds for a zero-order stochastic algorithm are achieved for minimizing a Gaussian 
smoothed approximation, $f_{\sigma}(w)=\EE_{z}[f(w+\sigma z)]$, of a Lipschitz continuous function 
$f(w)$, where $z$ follows a Normal distribution and $\sigma\geq 0$ is a smoothing parameter. In each 
iteration, the algorithm takes a step in the direction of an unbiased estimate of $\nabla f_{\sigma}(w)$. A computational complexity in terms of the number of iterations (or gradient estimates) of $O(\frac{1}{\delta \epsilon^2})$ is established for finding a 
solution satisfying $\EE(||\nabla f_\sigma(\bar{w})||^2_2)\leq 
\epsilon$, where $|f_{\sigma}(w)-f(w)|\leq \delta$ for all $w\in\RR^d$.  
Gaussian smoothing is a class of mollifiers with unbounded support, see for example \cite{ermoliev1995}, where it is established that these types of averaged functions, when taking $\sigma\rightarrow 0$, converge to $f$ when $f$ is continuous, and can preserve infima even in cases where $f$ is discontinuous.\\  

The recent publication \cite{zhang2020} seems to have been inspired by much of the same past research as our paper, resulting in similarities, most notably in the development of the same convergence criteria using the Clarke $\epsilon$-subdifferential. One of the biggest differences compared to our work is their assumption that the objective is directionally differentiable, resulting in a different analysis and algorithm design. In the stochastic setting, it is further assumed that an unbiased stochastic subgradient $\hat{g}$ can be sampled, such that $\EE[\hat{g}]=g\in \partial f(w)$ and $\langle g,d\rangle=f'(w,d)$, where $f'(w,d)$ denotes the directional derivative. Two algorithms are presented in their work. In the deterministic setting where $f(w)$ and a subgradient can be computed, an algorithm is presented with a convergence result in probability. In the stochastic setting, a convergence result in expectation is given. Using Markov's inequality and rerunning the algorithm a logarithmic number of times, one of the solutions will be an approximate stationary point with high probability. A similar computational complexity to 
ours has been proven for their algorithm. It is also notable that unlike our convergence results and those of \cite{nesterov2017}, their results are independent of the dimension size of the problem. A more detailed comparison of our convergence results is given in Subsection \ref{comp}. Lipschitz continuous functions are generally not directionally differentiable, but with this further assumption directional stationarity could be considered, which gives a sharper definition of a stationary point than Clarke or even a Mordukhovich stationary point. For algorithms converging to directional stationary points for non-smooth non-convex optimization problems, see for example \cite{pang2018}.

\section{Preliminaries}
\label{pre}
\subsection{Assumptions and Notation}
We assume that the random vector $\xi: \Omega\rightarrow \RR^p$ is a $(\F,\B_{\RR^p})$-measurable 
function where $\B_{\RR^p}$ denotes the Borel $\sigma$-algebra on $\RR^p$, and that $F(w,\xi)$ is a Carath\'eodory function \cite[Definition 4.50]{charalambos2013}, which in our setting means that for each $w\in \RR^d$, $F(w,\cdot)$ is $(\B_{\RR^p},\B_{\RR})$-measurable and for each $\xi\in \RR^p$, $F(\cdot, \xi)$ is continuous. It follows that $F(w,\xi)$ is a 
$(\B_{\RR^{d+p}},\B_{\RR})$-measurable function \cite[Lemma 4.51]{charalambos2013}. The function 
$F(w,\xi)$ is also assumed to be $C(\xi)$-Lipschitz continuous in $w$ for almost every $\xi$, implying that 
$f(w)$ is $L_0:=\EE[C(\xi)]$-Lipschitz continuous. In particular, the values of $\xi$ for which $F(w,\xi)$ is not Lipschitz continuous are contained in a Borel null set, whose complement of full measure is denoted as $\Xi$. We further assume that $Q:=\EE[C(\xi)^2]<\infty$. 
For a random variable $X: \Omega\rightarrow \RR^q$ for arbitrary $q\in \NN$, let $P_{X}$ denote the image measure induced by the random variable $X$, i.e. for a Borel set $A\in \B_{\RR^{q}}$, $P_{X}(A)=P(\{\omega\in\Omega: X(\omega)\in A\})$. Given that $\EE[C(\xi)^2]<\infty$, it holds that $C(\xi)\in L^1(P_{\xi})$
\cite[Proposition 6.12]{folland1999}.\\ 

Our analysis relies on randomly perturbed iterates, $w=x+z$, where 
$x\in\RR^d$ represents the current iterate and $z: \Omega\rightarrow \RR^d$ is a random vector uniformly 
distributed in the 
$d$-dimensional Euclidean ball of radius $\sigma>0$, $B(\sigma):=\{z:\norm{z}_2\leq \sigma\}$, denoted as $z\sim U(B(\sigma))$. The probability density function of $z$ is 
$$p(z)=\begin{cases}
\frac{\Gamma(\frac{d}{2}+1)}{(\sqrt{\pi}\sigma)^d} & \text{ if } z\in B(\sigma)\\ 
0 & \text{ otherwise,}  
\end{cases}\nonumber$$
where for $d\in \NN$, $\Gamma(\frac{d}{2}+1)=(\frac{d}{2})!$ when $d$ is even,  $\Gamma(\frac{d}{2}+1)=2^{-\left(\frac{d+1}{2}\right)}\sqrt{\pi}d!!$ when $d$ is odd. The double factorial for $d>0$ equals $d!!=d(d-2)(d-4)...(2 \text{ or } 1)$ for $d$ even or odd, and $0!!=1$. The expected distance from the origin of $z$ is 
\begin{alignat}{6}\label{eq:21}
\EE[\norm{z}_2]=\frac{\sigma d}{d+1}.
\end{alignat}

\subsection{Convergence Criterion}

The fact that $f(w)$ is Lipschitz continuous implies that it is differentiable everywhere outside of a 
set of Lebesgue measure zero due to Rademacher's theorem \cite[Theorem 3.1.6]{federer199}. This motivates the use of perturbed 
iterates as $f(w)$ is then differentiable with probability 1 whenever its gradient is evaluated at 
$w=x+z$. The proposed convergence criteria in this paper use the Clarke  $\epsilon$-subdifferential. 
We first 
define the Clarke subdifferential, which for locally Lipschitz continuous functions on $\RR^d$ equals
\begin{alignat}{6}  
\partial f(w):=\text{co}\{\lim\limits_{i\rightarrow \infty} \nabla f(w_i): w_i\rightarrow 
w,w_i\notin E\cup E_f\}, \label{clsub}
\end{alignat}
where $\text{co}\{\cdot\}$ denotes the convex hull, $E$ is any set of Lebesgue measure $0$, and $E_f$ is the set of points at which $f$ is not differentiable \cite[Theorem 2.5.1]{clarke1990}. The standard first-order convergence criterion for smooth non-convex functions is  
$\epsilon$-stationarity,
\begin{alignat}{6}  
\dist(0,\partial f(w))\leq \epsilon,\label{epsta} 
\end{alignat}
but as highlighted in the introduction, proving a convergence rate to such a point for Lipschitz 
continuous functions is not possible \cite[Theorem 5]{zhang2020}. This leads us to consider the 
Clarke $\epsilon$-subdifferential, 
\begin{alignat}{6}  
\partial_{\epsilon} f(w):=\text{co}\{\partial f(\hat{w}): \hat{w}\in w+B(\epsilon)\},\label{clesub} 
\end{alignat}
which is always a nonempty convex compact set with $\partial_{0} f(w)=\partial f(w)$ 
\cite{goldstein1977}.\\ 

The gradient is always contained in the Clarke subdifferential wherever a Lipschitz continuous 
function is differentiable \cite[Proposition 2.2.2]{clarke1990}, which motivates the use of the Clarke 
$\epsilon$-subdifferential, as for $z\sim U(B(\sigma))$ almost surely, $\nabla f(w)\in 
\partial_{\sigma}f(x)$ for $w=x+z$. Our focus is then on what we call 
$(\epsilon_1,\epsilon_2)$-stationarity,
\begin{alignat}{6}
\dist(0,\partial_{\epsilon_1}f(w))\leq \epsilon_2,\label{epss}
\end{alignat}	
with the goal of designing a stochastic algorithm which can output a random solution $\bar{w}$ 
which is an $(\epsilon_1,\epsilon_2)$-stationary point in 
expectation,     
\begin{alignat}{6}
\EE[\dist(0,\partial_{\epsilon_1}f(\bar{w}))]&\leq \epsilon_2.\label{eps}
\end{alignat}	
Relaxing the $\epsilon$-stationarity condition to finding a point which is a distance $\epsilon_1$ away from an $\epsilon_2$-stationary point, i.e. a point $\bar{w}$ such that 
\begin{alignat}{6}
||\bar{w}-\hat{w}||_2\leq \epsilon_1\quad\text{ and }\quad\dist(0,\partial f(\hat{w}))&\leq \epsilon_2,\label{nsp}
\end{alignat}	 
is an increasingly common convergence criterion for non-smooth non-convex objective functions, see for example \cite{davis2019,xu2019}. It has been shown though in \cite[Proposition 1]{kornowski2021} that \eqref{epss} does not imply \eqref{nsp}, but we can see that \eqref{epss} is a necessary condition for \eqref{nsp} to hold. 
In addition, \eqref{epss} has recently been used as a convergence criteria in \cite{zhang2020}, where the equivalence between  
$(\epsilon_1,\epsilon_2)$-stationarity and $\epsilon$-stationarity for functions with Lipschitz 
continuous gradients is shown in \cite[Proposition 6]{zhang2020}. For further discussion and examples of $(\epsilon_1,\epsilon_2)$-stationary points, see \cite{kornowski2021}.

\subsection{Differentiability and Measurability Properties}

The following property states that the gradient of $F(w,\xi)$ exists almost everywhere. Let $m^d$ denote the Lebesgue measure restricted to $(\RR^d,\B_{\RR^d})$. The product measure $m^d\times P_{\xi}$ is unique given that $m^d$ and $P_{\xi}$ are $\sigma$-finite. 
\begin{prop}{\cite[Lemma 1]{bianchi2020}}
	\label{AAder}
	The stochastic function $F(w,\xi)$ is differentiable in $w$ almost everywhere on the product measure space 
	${(\RR^{d+p},\B_{\RR^{d+p}},m^d\times P_{\xi})}$.	
\end{prop}

In order to handle the non-differentiability of $f(w)$, we define an {\it approximate gradient} of 
$f(w)$, $\widetilde{\nabla} f(w)$, to be a Borel measurable function on $\RR^d$ which equals the 
gradient of $f(w)$ almost everywhere it is differentiable. We also define the approximate stochastic 
gradient of $F(w,\xi)$, $\widetilde{\nabla} F: \RR^{d+p}\rightarrow \RR^d$, as a Borel measurable function which is equal to the gradient of $F(w,\xi)$ almost everywhere it exists. For completeness, an example of how to generate a family of approximate stochastic gradient functions $\widetilde{\nabla} F(w,\xi)$ is given in the appendix, 
see Example \ref{appgrad}. As the function $f(w)$ is continuous, it is Borel measurable 
and the approximate gradient $\widetilde{\nabla} f(w)$ can be constructed in the same manner as presented in 
the example.\\ 

The following property shows that unbiased estimates of 
the approximate gradient of $f(w)$ can be obtained by sampling the approximate stochastic gradient of 
$F(w,\xi)$ for almost all $w\in\RR^d$. 
\begin{prop}
	\label{expg}
	For almost every $w\in\RR^d$   	
	\begin{alignat}{6}
	&\widetilde{\nabla} f(w)=\EE_{\xi}[\widetilde{\nabla} F(w,\xi)],\nonumber
	\end{alignat}	
	where $\widetilde{\nabla} f(w)$ and $\widetilde{\nabla} F(w,\xi)$ are approximate gradients of 
	$f(w)$ and 
	$F(w,\xi)$, 
	respectively. 
\end{prop}

In the following property the measurability of $\dist(0,\partial_{\epsilon}f(w))$ is verified.
\begin{prop}
	\label{SetValMeas}
	For any $\epsilon\geq 0$, $\dist(0,\partial_{\epsilon}f(w))$ is a Borel measurable function in $w\in\RR^d$.
\end{prop}

Property \ref{AAder} was proven in \cite{bianchi2020}, but for completeness we give a proof in the appendix, along with the proofs of 
Properties \ref{expg} and \ref{SetValMeas}.

\section{Perturbed Iterate SGD}
\label{ssd}

\subsection{Algorithm Overview}

We now present PISGD. In each iteration $k$, $S$ perturbed values, $w^k_l$, of 
the current iterate $x^k$ are generated, and $S$ samples $\xi^k_l$ are taken for $l=1,...,S$. The 
stochastic function's approximate gradient is evaluated at each pair $(w^k_l,\xi^k_l)$ 
to generate the step direction, where all sampling is done independently. Our analysis 
assumes that the perturbation level $\sigma$ and step size $\eta$ are constant. As it is generally 
difficult to analyze the convergence of the last iterate of a stochastic algorithm, we use the 
standard technique \cite{ghadimi2013} of analyzing the average performance of the algorithm, which is 
equivalent to examining the convergence of a randomly chosen iterate $R$ out of a predetermined total 
of $K$. 
\begin{algorithm}[H]
	\caption{Perturbed Iterate SGD (PISGD)}
	\label{alg:sgd}
	\begin{algorithmic}
		\STATE {\bfseries Input:} $x^{1}\in \RR^d$; $K,S\in\ZZ_{>0}$; $\eta,\sigma >0$	
		\STATE $R\sim \text{uniform}\{1,2,...,K\}$
		\FOR{$k=1,2,...,R-1$} 
		\STATE Sample $z^k_l\sim U(B(\sigma))$ for $l=1,...,S$ 
		\STATE $w^k_l=x^k+z^k_l$ for $l=1,...,S$
		\STATE Sample $\xi^k_l\sim P_{\xi}$ for $l=1,...,S$ 
		\STATE $x^{k+1}=x^k-\frac{\eta}{S}\sum_{l=1}^S\widetilde{\nabla} F(w_l^k,\xi_l^k)$ 	
		\ENDFOR
		\STATE {\bfseries Output:} $x^R$
	\end{algorithmic}
\end{algorithm}

In the following subsections we give a number of results concerning PISGD which are now summarized:

\begin{itemize}
\item Subsection \ref{nonasympt} gives our main non-asymptotic convergence result to an expected $(\epsilon_1,\epsilon_2)$-stationary point in Theorem \ref{SGDF}.
\item Subsection \ref{comp} provides the computational complexity in terms of the number of stochastic approximate gradient computations for a range of algorithm settings in Corollary \ref{ssdcomp}, as well as for an optimized setting in Corollary \ref{ssdcomp2} for an expected $(\epsilon_1,\epsilon_2)$-stationary point. Corollary \ref{ssdcomp3} provides the computational complexity for an $(\epsilon_1,\epsilon_2)$-stationary point with probability $(1-\gamma)$ for any $\gamma\in(0,1)$.
\item Subsection \ref{asympt} contains a description of a general algorithm using PISGD which has asymptotic convergence guarantees to a Clarke stationary point proven in Corollary \ref{asymptotic}.
\item In Subsection \ref{settings}, the problem setting of $f(w)$ being a deterministic Lipschitz continuous function without any stochastic structure is considered, as well as the case where $f(w)$ takes the form of a finite-sum problem as in \eqref{eq:0}. 
\end{itemize}  

\subsection{Non-asymptotic Convergence to an Expected $(\epsilon_1,\epsilon_2)$-Stationary Point}
\label{nonasympt}
The following theorem provides guarantees for the values of $\epsilon_1$ and $\epsilon_2$ such that 
PISGD converges to an $(\epsilon_1,\epsilon_2)$-stationary point in expectation. Making $K$ large 
enough, $\epsilon_1$ and $\epsilon_2$ can be made arbitrarily small. There is also a parameter 
$\beta\in(0,1)$ which adjusts the rate of convergence in terms of $\epsilon_1$ and $\epsilon_2$. For example, taking $\beta=\frac{1}{3}$ the guarantees for $\epsilon_1$ and $\epsilon_2$ both improve at the same rate of $O(K^{-\beta})$.

\begin{theorem}	
	\label{SGDF}	
	Let $K\in \ZZ_{>0}$, $S=\lceil K^{1-\beta}\rceil$ for $\beta\in (0,1)$, 
	$\sigma=\theta\sqrt{d}K^{-\beta}$ and 
	$\eta=\frac{\theta}{L_0}K^{-\beta}$ for $\theta>0$, $Q=\EE[C(\xi)^2]$, and 
	$\Delta=f(x^1)-f(x^*)$, 
	where $f(x^*)$ is the global minimum of $f(x)$. After running 
	PISGD using an approximate stochastic gradient $\widetilde{\nabla} F(w,\xi)$,
	\begin{alignat}{6}
	\EE[\dist(0,\partial_{\sigma}f(x^R))]&&<&	
	K^{\frac{\beta-1}{2}}\sqrt{2\left(\frac{L_0}{\theta}\Delta+L_0^2\sqrt{d} 
	K^{-\beta}+Q\right)}.\label{bound}
	\end{alignat}	
\end{theorem}

Considering a standard implementation of mini-batch SGD with any chosen $K$, $S>1$, and $\eta$, PISGD can be implemented by choosing $\sigma=\eta L_0\sqrt{d}$ (using $\theta=L_0 \eta K^{\beta}$), and convergence guarantees can be computed in terms of an expected $(\epsilon_1,\epsilon_2)$-stationary point. The parameter $\theta>0$ allows for convergence guarantees to be made for any positive step size. The sample size has been fixed in the theorem, but the theorem holds for any $S\geq K^{1-\beta}$ (see equation 
\eqref{frhs}), so a valid $\beta$ always exists for any choice of $S>1$ and $K$. 
We also mention that following Property \ref{AAder} and the fact that the approximate stochastic gradient of $F(w^k_l,\xi^k_l)$ is evaluated a countable number of times, the probability of encountering a point of non-differentiability running Algorithm \ref{alg:sgd} is zero. To gain some intuition of how the right-hand side of \eqref{bound} is made small enough to ensure an expected $(\epsilon_1,\epsilon_2)$-stationary point, it can be replaced by the right-hand side of inequality \eqref{intuit} in the proof, which shows that $K$ needs to be made large enough to overcome problem specific constants and the choice of smoothing, which is less than or equal to $\epsilon_1$, and the sample size $S$ needs to be made large enough to make the variance of the stochastic step direction sufficiently small.\\

The proof of Theorem \ref{SGDF} requires the following three 
lemmas. The proofs can be found in the appendix.

\begin{lemma}	
	\label{ineqsmooth}	
	For $\{x,x',z\}\in \RR^d$, let $w=x+z$ and $w'=x'+z$. For a Lipschitz continuous function 
$f(\cdot)$ 
with approximate gradient $\widetilde{\nabla} f(\cdot)$, and any $x,x'\in \RR^d$,
\begin{alignat}{6}	
&f(w)-f(w')-\langle \widetilde{\nabla} f(w'),x-x'\rangle=\medint\int_0^1\langle\widetilde{\nabla} 
f(w'+v(x-x'))-\widetilde{\nabla} 
f(w'),x-x'\rangle dv\nonumber
\end{alignat}	
holds for almost all $z\in\RR^n$. 	
\end{lemma}

\begin{lemma}	
	\label{gradbound}	
	The norms of the approximate gradients are bounded, with
${\norm{\widetilde{\nabla} f(w)}_2\leq L_0}$ and 
$\norm{\widetilde{\nabla} F(w,\xi)}_2\leq C(\xi)$ almost everywhere.
\end{lemma}

\begin{lemma}	
	\label{stobound}	
	Let $x\in\RR^d$ and $z\in\RR^d$ be random variables, where $z$ is absolutely continuous, and $x$, $z$, and $\xi$ (as previously defined) are mutually independent. For any $S\in \ZZ_{>0}$, let  $$\overline{\nabla}F:=\frac{1}{S}\sum_{l=1}^S\widetilde{\nabla}	F(x+z_l,\xi_l),$$ 	
	where $z_l\sim P_z$ and $\xi_l\sim P_{\xi}$ for $l=1,...,S$. It holds that  
	\begin{alignat}{6}
	&&&\EE[\norm{\overline{\nabla}F}^2_2-\norm{\EE[\overline{\nabla}F|x]}^2_2]\nonumber\\
	&&=&\EE[\norm{\overline{\nabla}F-\EE[\overline{\nabla}F|x]}^2_2]\nonumber\\
	&&\leq&\frac{Q}{S},\nonumber
	\end{alignat}
	where $Q:=\EE[C(\xi)^2]$.
\end{lemma}

\begin{proof}[Proof of Theorem \ref{SGDF}.] 	
	Assume PISGD is run for $K$ iterations instead of 
	$R-1$, and for simplicity let 
		$$\overline{\nabla}F^k:=\frac{1}{S}\sum_{l=1}^S\widetilde{\nabla} F(w_l^k,\xi_l^k)$$ 
	for $k=1,...,K$, which is the random direction PISDG moves in each iteration. For any two iterates $x^k$ and $x^{k+1}$, let $\hat{w}^k=x^k+\hat{z}^k$ and 
	$\hat{w}^{k+1}=x^{k+1}+\hat{z}^k$ for a single sample $\hat{z}^k\sim 
	U(B(\sigma))$. As $\hat{z}^k$ is sampled uniformly from a Euclidean ball in $\RR^n$, 
	Lemma \ref{ineqsmooth} can be applied such that given $x^k$ and $x^{k+1}$, for almost all 
	values of $\hat{z}^k$,
	\begin{alignat}{6}
	&&&f(\hat{w}^{k+1})-f(\hat{w}^k)-\langle\widetilde{\nabla} 
	f(\hat{w}^k),x^{k+1}-x^k\rangle\nonumber\\
	&&=&\medint\int_0^1\langle\widetilde{\nabla} f(\hat{w}^k+v(x^{k+1}-x^k))-\widetilde{\nabla} 
	f(\hat{w}^k),x^{k+1}-x^k\rangle 
	dv,\label{bigeq}
	\end{alignat}
	since the distribution of $\hat{z}^k$ is absolutely continuous with respect to the Lebesgue measure.
	Considering now the expectation of \eqref{bigeq},
	\begin{alignat}{6}
	&&&\EE[f(\hat{w}^{k+1})-f(\hat{w}^k)-\langle\widetilde{\nabla} 
	f(\hat{w}^k),x^{k+1}-x^k\rangle]\nonumber\\
	&&=&\EE(\EE[f(\hat{w}^{k+1})-f(\hat{w}^k)-\langle\widetilde{\nabla} 
	f(\hat{w}^k),x^{k+1}-x^k\rangle|x^{k+1},x^k])\nonumber\\	
	&&=&\EE(\EE[\medint\int_0^1\langle\widetilde{\nabla} f(\hat{w}^k+v(x^{k+1}-x^k))-\widetilde{\nabla} 
	f(\hat{w}^k),x^{k+1}-x^k\rangle 
	dv|x^{k+1},x^k])\nonumber\\
	&&=&\EE[\medint\int_0^1\langle\widetilde{\nabla} f(\hat{w}^k+v(x^{k+1}-x^k))-\widetilde{\nabla} 
	f(\hat{w}^k),x^{k+1}-x^k\rangle 
	dv].\nonumber
	\end{alignat}		
	
	Applying $x^{k+1}-x^k=-\eta\overline{\nabla}F^k$, and given that for any $l\in 
	\{1,2,...,S\}$, $\EE[f(\hat{w}^{k+1})-f(\hat{w}^k)]=\EE[f(w_l^{k+1})-f(w_l^k)]$,
	\begin{alignat}{6}
	&&&\EE[f(w_l^{k+1})-f(w_l^k)+\eta\langle\widetilde{\nabla} 
	f(\hat{w}^k),\overline{\nabla}F^k\rangle]\nonumber\\
	&&=&\EE[\medint\int_0^1\langle\widetilde{\nabla} 
	f(\hat{w}^k-v\eta\overline{\nabla}F^k)-\widetilde{\nabla} f(\hat{w}^k),-\eta 	\overline{\nabla}F^k\rangle dv].\label{ineq2s}
	\end{alignat}

	Our first goal is to bound each side of \eqref{ineq2s} in terms of only $f(w_l^{k+1})-f(w_l^k)$ 
	and $\overline{\nabla}F^k$. We will 	
	analyze each side of \eqref{ineq2s} separately, then combine the analysis to prove the convergence 
	of the algorithm.\\  
		
	\noindent\textbf{Analysis of the left-hand side of \eqref{ineq2s}:}\\
	
	For all $k\in[1,...,K]$,
	\begin{alignat}{6}
	&&&\EE[f(w_l^{k+1})-f(w_l^k)+\eta\langle\widetilde{\nabla} 
	f(\hat{w}^k),\overline{\nabla}F^k\rangle]\nonumber\\
	&&=& \EE[f(w_l^{k+1})-f(w_l^k)]+\eta\EE(\EE[\langle\widetilde{\nabla} 
	f(\hat{w}^k),\overline{\nabla}F^k\rangle|x^k])\nonumber\\
	&&=&\EE[f(w_l^{k+1})-f(w_l^k)]+\eta\EE(\langle\EE[\widetilde{\nabla} 
	f(\hat{w}^k)|x^k],\EE[\overline{\nabla}F^k|x^k]\rangle).\label{cind}
	\end{alignat}
	The last equality holds since $\widetilde{\nabla} f(\hat{w}^k)$ and 
	$\overline{\nabla}F^k$ are conditionally independent random 
	variables with respect to $x^k$, so for all $j=1,..,d$,\footnote{Equalities involving conditional expectations are to be interpreted as holding almost surely.} $\EE[\widetilde{\nabla}_j 
	f(\hat{w}^k)\cdot\overline{\nabla}_jF^k|x^k]=\EE[\widetilde{\nabla}_j 
	f(\hat{w}^k)|x^k]\cdot\EE[\overline{\nabla}_jF^k|x^k]$. Focusing on 	
	$\EE[\overline{\nabla}F^k|x^k]$, 	
	\begin{alignat}{6}
	\EE[\overline{\nabla}F^k|x^k]&&=&\EE[\widetilde{\nabla} F(w^k_l,\xi_l^k)|x^k]\nonumber\\
	&&=&\EE[\widetilde{\nabla} f(w_l^k)|x^k]\label{tuffeq}\\
	&&=&\EE[\widetilde{\nabla} f(\hat{w}^k)|x^k],\label{dcond}
	\end{alignat}	
	for an arbitrary $l\in[1,...,S]$. The equality \eqref{tuffeq} follows from Property \ref{expg}: First let  
	$$g(y):=\EE[\widetilde{\nabla} F(y+z^k_l,\xi_l^k)],$$
	then 
	$$\EE[\widetilde{\nabla} F(w^k_l,\xi_l^k)|x^k]=g(x^{k})$$	
	since $z^k_l$ and $\xi_l^k$ are independent of $x^k$, see for example \cite[Lemma 
	2.3.4]{shreve2004}. From Lemma \ref{gradbound}, $|\widetilde{\nabla}_j F(y+z^k_l,\xi_l^k)|\leq C(\xi_l^k)$, which implies that $\widetilde{\nabla}_j F(y+z^k_l,\xi_l^k)\in L^1(P_{z^k_l}\times P_{\xi_l^k})$ for any $y\in\RR^d$. Given that $z_l^k$ and $\xi_l^k$ are 
	independent, Fubini's theorem can be applied with  
	$$g(y)=\EE_{z_l^k}[\EE_{\xi_l^k}[\widetilde{\nabla} F(y+z^k_l,\xi_l^k)]].$$
	From Property \ref{expg}, $\EE_{\xi_l^k}[\widetilde{\nabla} 
	F(y+z^k_l,\xi_l^k)]=\widetilde{\nabla} f(y+z^k_l)$ for 
	almost all 
	$z_l^k$, hence $g(y)=\EE[\widetilde{\nabla} f(y+z^k_l)]$ with $g(x^k)=\EE[\widetilde{\nabla} 
	f(w_l^k)|x^k]$, again from the independence of $x^k$ and $z^k_l$.\\
	
	Using \eqref{dcond} in \eqref{cind},	
	\begin{alignat}{6}
		&&&\EE[f(w_l^{k+1})-f(w_l^k)+\eta\langle\widetilde{\nabla} 
	f(\hat{w}^k),\overline{\nabla}F^k\rangle]\nonumber\\
	&&=& \EE[f(w_l^{k+1})-f(w_l^k)]+\eta\EE(\norm{\EE[\overline{\nabla}F^k|x^k]}^2_2).\label{eq:2}
	\end{alignat}
	
	\noindent\textbf{Analysis of the right-hand side of \eqref{ineq2s}:}\\
	
	\noindent We now analyze 
	\begin{alignat}{6}
	\EE[\medint\int_0^1\langle\widetilde{\nabla} f(\hat{w}^k-v\eta\overline{\nabla}F^k)-\widetilde{\nabla} 
	f(\hat{w}^k),-\eta 
	\overline{\nabla}F^k\rangle dv]\nonumber	
	\end{alignat}
	for any $k\in [1,...,K]$. As the negation, addition, composition, and product of real-valued Borel measurable functions,	
	${\langle\widetilde{\nabla} 
		f\left(\hat{w}^k-v\eta\overline{\nabla}F^k\right)-\widetilde{\nabla} f(\hat{w}^k),-\eta\overline{\nabla}F^k \rangle}$ is a measurable function on	
	${(\RR^{d+S(d+p)}\times[0,1],\B_{\RR^{d+S(d+p)}}\otimes \B_{[0,1]})}$, where 
	$(\hat{w}^k,\{w_l^k\},\{\xi_l^k\})\in \RR^{d+S(d+p)}$ and $v\in[0,1]$. Given that 
	$\widetilde{\nabla} f(w)$ and 
	$\widetilde{\nabla} F(w,\xi)$ are bounded almost everywhere by Lemma \ref{gradbound}, and the probability measure and the Lebesgue measure that the expectation and integral are with respect to are both finite, the function is integrable and Fubini's theorem can be applied: 
	\begin{alignat}{6}
	&&&\EE[\medint\int_0^1\langle\widetilde{\nabla} 
	f\left(\hat{w}^k-v\eta\overline{\nabla}F^k\right)-\widetilde{\nabla} 
	f(\hat{w}^k),-\eta\overline{\nabla}F^k \rangle dv]\nonumber\\
	&=&&\medint\int_0^1\EE[\langle\widetilde{\nabla} 
	f(\hat{w}^k-v\eta\overline{\nabla}F^k)-\widetilde{\nabla} 
	f(\hat{w}^k),-\eta\overline{\nabla}F^k\rangle]dv\nonumber\\
	&=&&\medint\int_0^1\EE(\EE[\langle\widetilde{\nabla} 
	f(\hat{w}^k-v\eta\overline{\nabla}F^k)-\widetilde{\nabla} 
	f(\hat{w}^k),-\eta\overline{\nabla}F^k\rangle|\overline{\nabla}F^k,x^k])dv\nonumber\\
	&=&&\medint\int_0^1\EE[\langle\EE[\widetilde{\nabla} 
	f(\hat{w}^k-v\eta\overline{\nabla}F^k)-\widetilde{\nabla} 
	f(\hat{w}^k)|\overline{\nabla}F^k,x^k],-\eta\overline{\nabla}F^k\rangle]dv\nonumber\\
	&\leq&&\medint\int_0^1\EE[\norm{\EE[\widetilde{\nabla} 
	f(\hat{w}^k-v\eta\overline{\nabla}F^k)-\widetilde{\nabla} 
		f(\hat{w}^k)|\overline{\nabla}F^k,x^k]}_2\cdot\norm{-\eta\overline{\nabla}F^k}_2]dv.\label{eq:3}
	\end{alignat}
	
	Focusing on $\widetilde{\nabla} 
		f(\hat{w}^k-v\eta\overline{\nabla}F^k)-\widetilde{\nabla} 
		f(\hat{w}^k)$ within $\EE[\widetilde{\nabla} 
	f(\hat{w}^k-v\eta\overline{\nabla}F^k)-\widetilde{\nabla} 
	f(\hat{w}^k)|\overline{\nabla}F^k,x^k]$, and writing $\hat{w}^k=x^k+\hat{z}^k$, the only random variable which is not measurable with respect to the $\sigma$-algebra generated by $\overline{\nabla}F^k$ and $x^k$, is $\hat{z}^k$, which is independent of $\overline{\nabla}F^k$ and $x^k$. 
	Similar to showing \eqref{tuffeq}, letting 
	\begin{alignat}{6}\label{dummy2}
	g(y,y'):=\EE[\widetilde{\nabla} f(y+\hat{z}^k)-\widetilde{\nabla}f(y'+\hat{z}^k)],
	\end{alignat}
	then 
	$$\EE[\widetilde{\nabla}					f(\hat{w}^k-v\eta\overline{\nabla}F^k)-\widetilde{\nabla}			f(\hat{w}^k)|\overline{\nabla}F^k,x^k]=g(x^k-v\eta\overline{\nabla}F^k,x^k).$$	
	Considering the norm of \eqref{dummy2} for 
	arbitrary $y,y'\in\RR^d$, and setting $w=y+\hat{z}^k$ and $w'=y'+\hat{z}^k$,
	\begin{alignat}{6}
	\norm{\EE[\widetilde{\nabla} f(w)-\widetilde{\nabla} f(w')]}_2&=&&\norm{\medint\int_{\RR^d} 
	\widetilde{\nabla} 
		f(w)p(w-y)dw-\medint\int_{\RR^d}\widetilde{\nabla} 
		f(w')p(w'-y')dw'}_2\nonumber\\
	&=&&\norm{\medint\int_{\RR^d} \widetilde{\nabla} f(w)(p(w-y)-p(w-y'))dw}_2\nonumber\\ 
	&\leq&&\medint\int_{\RR^d}\norm{\widetilde{\nabla}	f(w)}_2|p(w-y)-p(w-y')|dw\nonumber\\ 
	&\leq&&L_0\medint\int_{\RR^d} |p(w-y)-p(w-y')|dw\nonumber\nonumber\\
	&\leq&&L_0\frac{\lambda(d)}{\sigma}\frac{d!!}{(d-1)!!}\norm{y-y'}_2\nonumber\\
	&\leq&&L_0\frac{\sqrt{d}}{\sigma}\norm{y-y'}_2,\label{EE}
	\end{alignat}
	where the second inequality follows from Lemma \ref{gradbound} and the third inequality bounding 
	the integral $\medint\int_{\RR^d} |p(w-y)-p(w-y')|dw$ can be found in the 
	proof of Lemma 8 in \cite{yousefian2011} beginning at equation (33), where $\lambda(d)=\frac{2}{\pi}$ when $d$ is even and $1$ when $d$ is odd. The final inequality uses 
	the bound $\frac{\lambda(d)d!!}{(d-1)!!}\leq\sqrt{d}$, which is proven in Property \ref{boundprop} 
	in the appendix. The bound \eqref{EE} evaluated at $y=x^k-v\eta\overline{\nabla}F^k$ and $y'=x^k$ 
	gives 
	\begin{alignat}{6}
	&&&\norm{\EE[\widetilde{\nabla} f(\hat{w}^k-v\eta\overline{\nabla}F^k)-\widetilde{\nabla} 
		f(\hat{w}^k)|\overline{\nabla}F^k,x^k]}_2\leq 
		L_0\frac{\sqrt{d}}{\sigma}\norm{-v\eta\overline{\nabla}F^k}_2.\nonumber
	\end{alignat}	
	Applying this bound in \eqref{eq:3},	
	\begin{alignat}{6}
	&&&\EE[\medint\int_0^1\langle\widetilde{\nabla} 
	f\left(\hat{w}^k-v\eta\overline{\nabla}F^k\right)-\widetilde{\nabla} 
	f(\hat{w}^k),-\eta\overline{\nabla}F^k \rangle dv]\nonumber\\
	&\leq&&\medint\int_0^1\EE(L_0\frac{\sqrt{d}}{\sigma}\norm{-v\eta\overline{\nabla}F^k}_2\cdot\norm{-\eta\overline{\nabla}F^k}_2)dv\nonumber\\
	&=&&\medint\int_0^1L_0\frac{\sqrt{d}}{\sigma}\eta^2v\EE[\norm{\overline{\nabla}F^k}^2_2]dv\nonumber\\
	&=&&L_0\frac{\sqrt{d}}{\sigma}\frac{\eta^2}{2}\EE[\norm{\overline{\nabla}F^k}^2_2].\label{eq:4}
	\end{alignat}	
	
	\noindent\textbf{Proving the convergence of PISGD:}\\
	
	\noindent We now combine the analysis of equation \eqref{ineq2s}, using \eqref{eq:2} and 
	\eqref{eq:4} to get, for all $k\in[1,...,K]$,  
	the inequality
	\begin{alignat}{6}
	&\EE[f(w_l^{k+1})-f(w_l^k)]+\eta\EE(\norm{\EE[\overline{\nabla}F^k|x^k]}^2_2)\nonumber\\
	=&\EE[f(w_l^{k+1})-f(w_l^k)+\eta\langle\widetilde{\nabla} 
	f(\hat{w}^k),\overline{\nabla}F^k\rangle]\nonumber\\
	=&\EE[\medint\int_0^1\langle\widetilde{\nabla} f(\hat{w}^k-v\eta\overline{\nabla}F^k)-\widetilde{\nabla} 
	f(\hat{w}^k),-\eta 
	\overline{\nabla}F^k\rangle dv]\nonumber\\ 
	\leq&L_0\frac{\sqrt{d}}{\sigma}\frac{\eta^2}{2}\EE[\norm{\overline{\nabla}F^k}^2_2].\nonumber
	\end{alignat}	
	
	Adding $\eta\EE[\norm{\overline{\nabla}F^k}^2_2]$ to both sides and rearranging:
	\begin{alignat}{6}	
	&&&\EE[f(w_l^{k+1})-f(w_l^k)]+\eta\EE[\norm{\overline{\nabla}F^k}^2_2]\nonumber\\	
	&&\leq&L_0\frac{\sqrt{d}}{\sigma}\frac{\eta^2}{2}\EE[\norm{\overline{\nabla}F^k}^2_2]+\eta\EE[\norm{\overline{\nabla}F^k}^2_2]
	-\eta\EE(\norm{\EE[\overline{\nabla}F^k|x^k]}^2_2)\nonumber\\
	&&=&L_0\frac{\sqrt{d}}{\sigma}\frac{\eta^2}{2}\EE[\norm{\overline{\nabla}F^k}^2_2]
	+\eta\EE[\norm{\overline{\nabla}F^k}^2_2
	-\norm{\EE[\overline{\nabla}F^k|x^k]}^2_2]\nonumber\\
	&&\leq&L_0\frac{\sqrt{d}}{\sigma}\frac{\eta^2}{2}\EE[\norm{\overline{\nabla}F^k}^2_2]+\eta
	 \frac{Q}{S},\label{Qbound}
	 \end{alignat}
	 using Lemma \ref{stobound}. Rearranging,
	\begin{alignat}{6}	 
	&&\left(\eta-L_0\frac{\sqrt{d}}{\sigma}\frac{\eta^2}{2}\right)\EE[\norm{\overline{\nabla}F^k}^2_2]\leq&\EE[f(w_l^k)-f(w_l^{k+1})]+\eta\frac{Q}{S}.\nonumber
	\end{alignat}
	Summing these inequalities for $k=1,...,K$,	
	\begin{alignat}{6}
	&&&\left(\eta-L_0\frac{\sqrt{d}}{\sigma}\frac{\eta^2}{2}\right)\sum_{k=1}^K\EE[\norm{\overline{\nabla}F^k}^2_2]\leq&\EE[f(w_l^1)-f(w_l^{K+1})]+\eta K\frac{Q}{S}.\nonumber
	\end{alignat}	
	As $R$ was sampled uniformly over $\{1,2,...,K\}$,
	\begin{alignat}{6}	
	&&\left(\eta-L_0\frac{\sqrt{d}}{\sigma}\frac{\eta^2}{2}\right)\EE[\norm{\overline{\nabla}F^R}^2_2]\leq&\frac{1}{K}\EE[f(w_l^1)-f(w_l^{K+1})]+\eta
	 \frac{Q}{S}\nonumber\\
	&&\leq&\frac{1}{K}(\EE[f(w_l^1)]-f(x^*))+\eta \frac{Q}{S}\nonumber\\
	&&\leq&\frac{1}{K}(f(x^1)+L_0\EE[\norm{z_l^1}_2]-f(x^*))+\eta \frac{Q}{S}\nonumber\\
	&&=&\frac{1}{K}(\Delta+L_0\frac{\sigma d}{d+1})+\eta \frac{Q}{S}.\nonumber
	\end{alignat}
	The last inequality uses the Lipschitz continuity of 
	$f(w)$, and the equality uses \eqref{eq:21} and sets $f(x^1)-f(x^*)=\Delta$. Taking 
	$\eta=\frac{\sigma}{L_0\sqrt{d}}$,
	\begin{alignat}{6}
	\EE[\norm{\overline{\nabla} F^R}^2_2]
	&&\leq&\frac{2}{K}L_0\frac{\sqrt{d}}{\sigma}(\Delta +L_0\frac{\sigma 
	d}{d+1})+2\frac{Q}{S}\label{intuit}\\
	&&<&2\frac{L_0}{K}\frac{\sqrt{d}}{\sigma}\Delta 
	+2\frac{L_0^2}{K}\sqrt{d}+2\frac{Q}{S}.\nonumber
	\end{alignat}	
	Setting $\sigma=\theta \sqrt{d}K^{-\beta}$ and $S\geq K^{1-\beta}$, e.g. $S=\lceil 
	K^{1-\beta}\rceil$, 
	\begin{alignat}{6}
	\EE[\norm{\overline{\nabla} F^R}^2_2]
	&&<&2 K^{\beta-1}\frac{L_0}{\theta}\Delta 
	+2\frac{L_0^2}{K}\sqrt{d}+2K^{\beta-1}Q\nonumber\\	
	&&=&2K^{\beta-1}\left(\frac{L_0}{\theta}\Delta+L_0^2\sqrt{d} 
	K^{-\beta}+Q\right),\label{frhs}
	\end{alignat}
	and 
	$\eta=\frac{\theta}{L_0}K^{-\beta}$. In addition,
	\begin{alignat}{6}
	\EE[\norm{\overline{\nabla} F^R}^2_2]
	&&=&\EE[\norm{\frac{1}{S}\sum\limits_{l=1}^S\widetilde{\nabla} 
		F(w_l^R,\xi_l^R)}^2_2]\nonumber\\
	&&=&\EE(\EE[\norm{\frac{1}{S}\sum\limits_{l=1}^S\widetilde{\nabla} 
		F(w_l^R,\xi_l^R)}^2_2|x^R])\nonumber\\	
	&&\geq&\EE(\norm{\EE[\frac{1}{S}\sum\limits_{l=1}^S\widetilde{\nabla} 
		F(w_l^R,\xi_l^R)|x^R]}^2_2)\nonumber\\
	&&=&\EE[\norm{\EE[\frac{1}{S}\sum_{l=1}^S\widetilde{\nabla} 
		f(w_l^R)|x^R]}^2_2]\label{hpineq}\\
	&&\geq&\EE[\dist(0,\partial_{\sigma}f(x^R))^2].\label{flhs}
	\end{alignat}
	The third equality follows from Property \ref{expg} like \eqref{tuffeq}. For all 	
	$l=1,...,S$, $w^R_l=x^R+z_l^R\in x^R+B(\sigma)$. 
	The gradient is always contained in the Clarke subdifferential wherever a Lipschitz continuous 
	function is differentiable, so for almost every $z_l^R$, 
	the 
	approximate gradient $\widetilde{\nabla} f(w^R_l)\in 
	\partial_{\sigma}f(x^R)$. The convex combination, 
	$\frac{1}{S}\sum_{l=1}^S\widetilde{\nabla} f(w^R_l)\in \partial_{\sigma}f(x^R)$ almost surely as 
	well, 
	given that $\partial_{\sigma}f(x^R)$ is a convex set, hence   
	$\EE[\frac{1}{S}\sum_{l=1}^S\widetilde{\nabla} f(w^R_l)|x^R]\in \partial_{\sigma}f(x^R)$, 
	resulting 
	in the final inequality. Combining \eqref{frhs} and \eqref{flhs}, and using Jensen's inequality, 
	\begin{alignat}{6}
	\EE[\dist(0,\partial_{\sigma}f(x^R))]&&<&	
	K^{\frac{\beta-1}{2}}\sqrt{2\left(\frac{L_0}{\theta}\Delta+L_0^2\sqrt{d} K^{-\beta}+Q\right)}.\nonumber
	\end{alignat}	
\end{proof}

\subsection{Computational Complexity}
\label{comp}

The following corollary establishes computational complexities for finding an 
$(\epsilon_1,\epsilon_2)$-stationary point in expectation in terms of the number of stochastic 
approximate gradient computations $\widetilde{\nabla} F(w,\xi)$, which we will simply refer to as 
gradient calls in this subsection. For example, choosing 
$\beta=\frac{1}{3}$, the complexity is 
$O(\min(\epsilon_1,\epsilon_2)^{-5})$ and for $\beta=\frac{1}{2}$, it is 
$O(\max(\epsilon_1^{-3},\epsilon_2^{-6}))$. 
\begin{corollary}	
	\label{ssdcomp}
	For $\beta\in(0,1)$, an expected $(\epsilon_1,\epsilon_2)$-stationary point \eqref{eps} can be 
	computed with 	
	$O\left(\max\left(\epsilon_1^{\frac{\beta-2}{\beta}},\epsilon_2^{-2\frac{2-\beta}{1-\beta}}\right)\right)$
	gradient calls.
\end{corollary}
The optimal choice for $\beta$ is somewhat ambiguous as it depends on the importance placed on 
$\epsilon_1$ and $\epsilon_2$. This can be resolved by attempting to find the optimal $\beta$, with respect to the bound \eqref{bound} provided by Theorem \ref{SGDF}, which minimizes the upper bound on the total number of gradient calls used in PISGD,  
$(K-1)S=(K-1)\lceil K^{1-\beta}\rceil$, assuming $K-1$ iterations are performed. By the Lipschitz continuity of $f(w)$, every point is an $(\epsilon_1,L_0)$-stationary point, so we will assume that $\epsilon_2<L_0$ for the rest of this subsection. In addition, the parameter $\theta$ in Theorem \ref{SGDF} is included to allow the theorem to be applicable for any step size $\eta>0$, but it is redundant for the convergence of the algorithm, so for simplicity we fix $\theta=1$ for the rest of this subsection. 
\begin{corollary}	
\label{ssdcomp2}
Assume $\epsilon_2<L_0$, $\theta=1$, and let 
\begin{alignat}{6}
K^*&=&\max\left(\left\lfloor \frac{2}{\epsilon^2_2}\left(
L_0\Delta+Q+\sqrt{d}L_0^2\right)+1\right\rfloor,\left\lceil \frac{2\sqrt{d}}{\epsilon^2_2}\left(
\frac{L_0\Delta+Q}{\epsilon_1}+L_0^2\right)\right\rceil\right)\nonumber
\end{alignat}
and  
\begin{alignat}{6}  
\beta^*=\frac{\log(K^*\epsilon^2_2-2\sqrt{d}L_0^2)-\log(2(L_0\Delta+Q))}{\log(K^*)},\nonumber
\end{alignat}
then 
\begin{enumerate}
	\item an expected $(\epsilon_1,\epsilon_2)$-stationary point \eqref{eps} can be 
	computed with 	
	$O\left(\frac{1}{\epsilon_1\epsilon^4_2}\right)$ gradient calls, 
	\item $(K^*,\beta^*)$ has the same gradient call complexity as a minimizer of the number of gradient calls required to find an expected 
	$(\epsilon_1,\epsilon_2)$-stationary point using inequality \eqref{bound} of Theorem \ref{SGDF}, and
	\item $(K^*,\beta^*)$ minimizes the number of iterations 
	required to find an expected $(\epsilon_1,\epsilon_2)$-stationary point using inequality \eqref{bound} of Theorem \ref{SGDF}.
\end{enumerate}
\end{corollary}

Motivated by \cite[Section 2.2]{ghadimi2013} we present the computational complexity for finding an $(\epsilon_1,\epsilon_2)$-stationary point with probability $1-\gamma$ for any $\gamma\in (0,1)$.

\begin{corollary}
\label{ssdcomp3}
Let $c\in(0,1)$ and $\phi>1$ be arbitrary constants. For any $\gamma\in(0,1)$ and $\epsilon_2<L_0$, let PISGD be run $\R:=\lceil-\ln(c\gamma)\rceil$ times using the parameter settings of Theorem \ref{SGDF} with $\theta=1$, 
\begin{alignat}{6}
K&=&\max\left(\left\lfloor \frac{2}{(\epsilon'_2)^2}\left(
L_0\Delta+Q+\sqrt{d}L_0^2\right)+1\right\rfloor,\left\lceil \frac{2\sqrt{d}}{(\epsilon'_2)^2}\left(
\frac{L_0\Delta+Q}{\epsilon_1}+L_0^2\right)\right\rceil\right),\nonumber
\end{alignat}
and  
\begin{alignat}{6}  
\beta=\frac{\log(K(\epsilon'_2)^2-2\sqrt{d}L_0^2)-\log(2(L_0\Delta+Q))}{\log(K)},\nonumber
\end{alignat}
where $\epsilon_2'=\sqrt{\frac{\epsilon^2_2-6\psi\frac{Q}{T}}{4e}}$, $\psi=\frac{\lceil-\ln(c\gamma)\rceil}{(1-c)\gamma}$, and $T=\lceil6\phi\psi\frac{Q}{\epsilon^2_2}\rceil$, outputting candidate solutions $\overline{X}:=\{\bar{x}^1,...,\bar{x}^{\R}\}$. With $T$ samples $\{(z_1,\xi_1),...,(z_{T},\xi_{T})\}$, where $z_i\sim U(B(\sigma))$ and $\xi_i\sim P_{\xi}$ for $i=1,...,T$, let $\bar{x}^*\in\overline{X}$ be chosen such that for $\overline{\nabla} F_{T}(x):=\frac{1}{T}\sum_{t=1}^{T}\widetilde{\nabla}F(x+z_t,\xi_t)$,
\begin{alignat}{6}
\bar{x}^*\in\argmin\limits_{x\in\overline{X}}||\overline{\nabla} F_{T}(x)||_2.\nonumber
\end{alignat} 
It follows that 
\begin{enumerate}
	\item $\bar{x}^*$ is an $(\epsilon_1,\epsilon_2)$-stationary point with a probability of at least $1-\gamma$, and 
	\item the described method requires $\tilde{O}\left(\frac{1}{\epsilon_1\epsilon^4_2}+\frac{1}{\gamma\epsilon^2_2}\right)$ gradient calls.
\end{enumerate}
\end{corollary}
The proofs of the three corollaries are contained in the appendix.\\
 
\noindent\textbf{Comparison with the computational complexity in \cite{zhang2020}:}\\
For the deterministic setting,  Interpolated Normalized Gradient Descent (INGD) is developed which has a gradient call complexity of $\tilde{O}(\frac{1}{\epsilon_1\epsilon_2^3})$ to achieve an $(\epsilon_1,\epsilon_2)$-stationary point with arbitrarily high probability.
For the stochastic setting, Stochastic INDG finds an expected $(\epsilon_1,\epsilon_2)$-stationary point with $\tilde{O}(\frac{1}{\epsilon_1\epsilon_2^4})$ stochastic subgradient calls. In particular, their algorithm finds an 
expected $(\epsilon_1,\frac{\epsilon_2}{3})$-stationary point. Running the algorithm $-\log(\gamma)$ times, one of the solutions will be an $(\epsilon_1,\epsilon_2)$-stationary point with a probability of at least $1-\gamma$. Our computational complexities are similar, with the extra $\tilde{O}\left(\frac{1}{\gamma\epsilon^2_2}\right)$ term in our high probability convergence result coming from returning an $\bar{x}^*\in\overline{X}$ instead of simply $\overline{X}$. We also point out that we have omitted the problem dimension $d$ in our computational complexity, where the convergence result of \cite{zhang2020} is {\it dimension-free}, which is a sought-after property when studying the computational complexity of algorithms, see for example \cite{carmon2019}. 

\subsection{Convergence to a Clarke Stationary Point Almost Surely}
\label{asympt}

Inspired by the discussion of an algorithm with asymptotic convergence to an $\epsilon$-stationary point \eqref{epsta} with high probability on page 4 of \cite{zhang2020}, in the following corollary, we prove that if PISDG is run for a sequence of increasing iteration sizes, any accumulation point of the solutions is a Clarke stationary point almost surely.

\begin{corollary}	
	\label{asymptotic}
	Let $\{K_1,K_2,...\}$ be a strictly increasing sequence of positive integers. Let $\beta\in (0,1)$ and a finite $\theta>0$ be fixed, with $(S_i,\sigma_i,\eta_i)$ equal to $(S,\sigma,\eta)$ as described in Theorem \ref{SGDF} given $(\beta,\theta,K=K_i)$. Assume for $i=1,2,...,$ PISGD is run with the indexed parameters, outputting solutions $\{x_1, x_2,...\}$, i.e. $x_i=x^R$ for the $i^{th}$ instance of running PISGD. Any accumulation point $x^*$ of the solutions $\{x_1, x_2,...\}$ is a Clarke stationary point almost surely.	
\end{corollary}

\begin{proof}
Assuming there exists an accumulation point $x^*$ of $\{x_i\}$, let $\{x_i\}$ be redefined as a subsequence of $\{x_i\}$ such that $\lim\limits_{i\rightarrow \infty} x_{i}=x^*$. 
Let
\begin{alignat}{6}
B_i:=K_i^{\frac{\beta-1}{2}}\sqrt{2\left(\frac{L_0}{\theta}\Delta+L_0^2\sqrt{d} K_i^{-\beta}+Q\right)},\nonumber
\end{alignat}
which is the right-hand side of \eqref{bound} with $K=K_i$. As $i\rightarrow \infty$, $\sigma_i\rightarrow 0$ and $B_i\rightarrow 0$. For any $i\in \NN$ there exists an $I\in \NN$ such that for all $j>I$, $x_j$ is an expected  
$\left(\frac{\sigma_i}{2},\frac{B_i}{i}\right)$-stationary point with 
$||x_{j}-x^*||_2\leq \frac{\sigma_i}{2}$. For such solutions $x_j$,
\begin{alignat}{6}  
\{\partial f(x): x\in x_j+B(\sigma_i/2)\}\subseteq\{\partial f(x): x\in x^*+B(\sigma_i)\},\nonumber
\end{alignat}
hence $\partial_{\frac{\sigma_i}{2}} f(x_{j})\subseteq \partial_{\sigma_i} f(x^*)$ and given that $x_j$ is an expected  
$\left(\frac{\sigma_i}{2},\frac{B_i}{i}\right)$-stationary point,
\begin{alignat}{6}
\EE[\dist(0,\partial_{\sigma_i}f(x^*))]&&\leq&	\frac{B_i}{i}.\nonumber
\end{alignat}
Using Markov's inequality,
\begin{alignat}{6}
\PP[\dist(0,\partial_{\sigma_i}f(x^*))\geq\frac{1}{i}]\leq B_i.\nonumber
\end{alignat}
Given that $\dist(0,\partial_{\sigma_i}f(x^*))\leq \dist(0,\partial_{\sigma_{i+1}}f(x^*))$ and $\frac{1}{i}>\frac{1}{i+1}$, the sets 
\begin{alignat}{6}
V_i:=\{x^*\in\RR^d:\dist(0,\partial_{\sigma_i}f(x^*))\geq \frac{1}{i}\}\nonumber
\end{alignat}
are monotonically increasing, $V_i\subseteq V_{i+1}$, and the limit $\lim\limits_{i\rightarrow \infty}V_i=\bigcup\limits_{i\geq 1}V_i$ exists \cite[Exercise 2.F.]{bartle1995}. Since the functions $\dist(0,\partial_{\sigma_i}f(x^*))$ are Borel measurable from Property \ref{SetValMeas}, each $V_i$ is Borel measurable, as is 
$\lim\limits_{i\rightarrow \infty}V_i$ as a countable union of Borel measurable sets.\\ 

The next step is to prove that $\lim\limits_{i\rightarrow \infty}V_i=\{x^*\in\RR^d:\dist(0,\partial f(x^*))>0\}$. For an $x\in \bigcup\limits_{i\geq 1}V_i$ there exists an $i\geq 1$ such that $\dist(0,\partial f(x))\geq \dist(0,\partial_{\sigma_i}f(x))\geq \frac{1}{i}>0$, hence $x\in \{x^*\in\RR^d:\dist(0,\partial f(x^*))>0\}$.\\ 

For an $x\in \{x^*\in\RR^d:\dist(0,\partial f(x^*))>0\}$, let $\omega=\dist(0,\partial f(x))$. The Clarke subdifferential is an upper semicontinuous set valued mapping \cite[Proposition 2.1.5 (d)]{clarke1990}, which means that for all $\omega_1>0$, there exists an $\omega_2>0$ such that $\partial f(\hat{x})\subset \partial f(x)+B(\omega_1)$ for all $\hat{x}\in x+B(\omega_2)$,\footnote{For our setting, using closed balls is equivalent to using open balls in the definition.} hence $\partial_{\omega_2} f(x)\subseteq \text{co}\{\partial f(x)+B(\omega_1)\}=\partial f(x)+B(\omega_1)$, and 
\begin{alignat}{6}
\dist(0,\partial_{\omega_2} f(x))\geq \dist(0,\partial f(x)+B(\omega_1)).\label{omega2dist}
\end{alignat}
The function $\dist(0,\partial f(x)+B(\omega_1))$ can be bounded below as 
\begin{alignat}{6}
\dist(0,\partial f(x)+B(\omega_1))&=\min_{\substack{z\in\partial f(x)\\y\in B(\omega_1)}}||z+y||_2\nonumber\\
&\geq \min_{\substack{z\in\partial f(x)\\y\in B(\omega_1)}}||z||_2-||y||_2\nonumber\\
&=\omega -\omega_1.\label{lowerbound}
\end{alignat}
Choosing $\omega_1=\frac{\omega}{2}$, there then exists an $\omega_2>0$ such that  
$\dist(0,\partial_{\omega_2} f(x))\geq \frac{\omega}{2}$ from \eqref{omega2dist} and \eqref{lowerbound}. A $J\in \NN$ exists such that for all $i\geq J$, $\sigma_i\leq\omega_2$, and setting $I\geq\max\{J,\lceil\frac{2}{\omega}\rceil\}$, $x\in V_i$ for all $i\geq I$, proving that $x\in \bigcup\limits_{i\geq 1}V_i$ and  $\lim\limits_{i\rightarrow \infty}V_i=\{x^*\in\RR^d:\dist(0,\partial f(x^*))>0\}$.\\

It follows that $\PP[\dist(0,\partial f(x^*))=0]=1$ as 
\begin{alignat}{6}
&&\PP[\dist(0,\partial f(x^*))>0]&=\lim\limits_{i\rightarrow\infty}\PP[\dist(0,\partial_{\sigma_i} f(x^*))\geq \frac{1}{i}]\nonumber\\
&&&\leq\lim_{i\rightarrow\infty}B_i=0,\nonumber
\end{alignat}
where the equality holds since $V_i\subseteq V_{i+1}$ \cite[Theorem A.1.1]{shreve2004}.
\end{proof}

\subsection{PISGD and Theorem \ref{SGDF} in Particular Cases}
\label{settings}

In this subsection we look at how the convergence result of Theorem \ref{SGDF} changes for particular forms of $f(w)$.\\

\noindent\textbf{Deterministic $f(w)$:}
In the case where $f(w)$ does not have the structure of \eqref{eq:1} and is simply a deterministic $L_0$-Lipschitz continuous function, there will be no sampling of $\xi$ with the algorithm update rule being
\begin{alignat}{6} 
x^{k+1}=x^k-\frac{\eta}{S}\sum_{l=1}^S\widetilde{\nabla} f(w_l^k).\label{deterup}
\end{alignat} 	
The only change in Theorem \ref{SGDF} is that $Q$ is no longer needed, with $L^2_0$ replacing it in inequality \eqref{bound}:
\begin{alignat}{6}
\EE[\dist(0,\partial_{\sigma}f(x^R))]&&<&	
K^{\frac{\beta-1}{2}}\sqrt{2L_0\left(\frac{\Delta}{\theta}+L_0\sqrt{d} 
K^{-\beta}+L_0\right)}.\nonumber
\end{alignat}
This comes from changing inequality \eqref{Qbound} in the proof of Theorem \ref{SGDF} after replacing  
\begin{alignat}{6}
\frac{1}{S}\EE[\norm{\widetilde{\nabla} F(w_l^k,\xi_l^k)}^2_2]
&&\leq&\frac{Q}{S}\nonumber
\end{alignat}
with
\begin{alignat}{6}
\frac{1}{S}\EE[\norm{\widetilde{\nabla} f(w_l^k)}^2_2]
&&\leq&\frac{L^2_0}{S},\nonumber
\end{alignat}
at inequality \eqref{eq:22} in the proof of Lemma \ref{stobound}, given that $\norm{\widetilde{\nabla} f(w_l^k)}_2\leq L_0$ almost surely by Lemma \ref{gradbound}.\\

\noindent\textbf{Finite-sum $f(w)$:}
For the case where $f(w)$ takes the finite-sum structure of \eqref{eq:0},
in mini-batch SGD, when the required sample size $S\geq n$, it is better to switch to gradient descent. In PISGD, if $S\geq n$, one could use the update rule \eqref{deterup}, but this would result in $n$ times the number of approximate stochastic gradient computations compared to continuing to use 
\begin{alignat}{6}
x^{k+1}=x^k-\frac{\eta}{S}\sum_{l=1}^S\widetilde{\nabla} F(w_l^k,\widehat{\xi}_l^k).\nonumber
\end{alignat}
Using the update \eqref{deterup} would replace $Q$ with $L_0^2$ in the bound of $\epsilon_2$, and it would no longer be required to uniformly sample $\{\widehat{\xi}\}$ each iteration, but the increased number of approximate stochastic gradient computations required is likely to outweigh these benefits.

\section{Application: Feedforward Neural Network}
\label{num}
In this section we consider training a fully connected 
feedforward neural network with one hidden layer using MNIST data, with 
$N=[68,9,3]$ nodes in each layer, respectively. The MNIST training 
dataset consists of image data $v^i$ for $i=1,...,60,000$, of the digits $0,1,...,9$, of dimension $784$ and one-hot encoded labels $y^i$ of dimension $10$. The neural network trained on the digits $[0,1,2]$, which consisted of $n=18624$ samples. PCA was applied to $v$ with 90\% explained variance,\footnote{This modified MNIST dataset is available from the corresponding author on reasonable request.} which reduced the  dimension of each $v^i$ to $p=68$.\\

The decision variables of the model are $x=[W,b]$, where for $l=2,3$, $W^l_{jk}$ is the weight for 
the connection between the $k^{th}$ neuron 
in the $(l-1)^{th}$ layer and the $j^{th}$ neuron in the $l^{th}$ layer, and $b^l_j$ is the bias of 
the $j^{th}$ neuron in the $l^{th}$ layer. The input and output of the activation functions in each 
layer are denoted as $z^l_j$ and $\alpha^l_j$, respectively. ReLU-m activation functions were used 
in the hidden layer, 
\begin{alignat}{6}
\alpha^2_j(z^2_j)&:=&&\min(\max(z^2_j,0),m),\nonumber
\end{alignat} 
with $m>0$, and softmax functions were used in the output 
layer, 
\begin{alignat}{6}
\alpha^3_j(z^3):=\frac{e^{z^3_j}}{\sum_{k=1}^{N_3}e^{z^3_k}},\nonumber
\end{alignat} 
with a cross-entropy loss function, 
\begin{alignat}{6}
\LL(\alpha^3,y^i):=-\sum_{j=1}^{N_3}y^i_j\log(\alpha^3_j).\nonumber
\end{alignat} 
All of the weights in $W^3$ were put through hard tanh activation functions,
\begin{alignat}{6}
H_{jk}(W^3_{jk})&:=&&\min(\max(W^3_{jk},-1),1).\nonumber
\end{alignat} 
The optimization problem is then
\begin{alignat}{6}
&\min\limits_{W,b}&&\text{ }\frac{1}{n}\sum_{i=1}^n 
\LL(\alpha^3(H(W^3)\alpha^2(W^2v^i+b^2)+b^3),y^i).\nonumber
\end{alignat}
Applying hard tanh activation functions directly to weights is similar to ideas such as weight 
normalization 
\cite{salimans2016} and using bounded-weights \cite{liao2004}. Our motivation 
to include these activation functions was to be able to compute a Lipschitz constant for 
$\LL$ and 
objectively test PISDG with parameters computed using Theorem \ref{SGDF}. The 
proof of the following property is in the appendix.  

\begin{prop}
	\label{nnl}
	Each function $\LL(\alpha^3(H(W^3)\alpha^2(W^2v^i+b^2)+b^3),y^i)$ 
	is\\ $L_i:=2\max(\sqrt{N_2N_3}\norm{[(v^i)^T,1]}_2,\sqrt{(N_2m^2+1)})$-Lipschitz continuous.
\end{prop}
By the Lipschitz continuity proved in Property \ref{nnl}, $\LL(\alpha^3(H(W^3)\alpha^2(W^2v^i+b^2)+b^3),y^i)$ is differentiable almost everywhere in $x=[W,b]$. However, applying the chain rule as done in backpropagation for deep learning models, even when Lipschitz continuous, is not generally valid for almost all $x$. For an in depth analysis of the validity of using auto differentiators, and in particular the backpropagation algorithm, see \cite{bolte2021}. In the following subsection, we show that using the chain rule to compute $\widetilde{\nabla} F(w_l^k,\xi_l^k)$ for the current application outputs the gradient with probability 1. 

\subsection{Using PISGD with the Chain Rule for Minimizing $\LL$}
Let $$\LL_i:=\LL(\alpha^3(H(W^3+z_{W^3})\alpha^2((W^2+z_{W^2})v^i+b^2+z_{b^2})+b^3+z_{b^3}),y^i)$$
where $z=[z_{W^3},z_{b^3},z_{W^2},z_{b^2}]\sim U(B(\sigma))^{N_3\times N_2+N_3+N_2\times N_1+N_2}$ is the required iterate perturbation used in PISGD. In this application, $\xi\in \RR^p$ is uniformly drawn from $(v,y)$, i.e. $\xi=(v^i,y^i)$ for any $i\in [1,2,...,n]$ with probability $\frac{1}{n}$. For simplicity, we omit iteration or sample notation. Our analysis holds considering  $[W,b]=x^k$, $[z_{W^3},z_{b^3},z_{W^2},z_{b^2}]=z^k_l$ and $(v^i,y^i)=(v^k_l,y^k_l)=\xi^k_l$ for any $k\in[1,...,K]$ and $l\in [1,...,S]$. Our approach is to show that using the chain rule for each decision variable $x'\in x$ produces the partial derivative $\frac{\partial \LL_i}{\partial x'}(x+z)$ with probability 1. Given that $\LL_i$ is Lipschitz continuous, this implies that the chain rule outputs $\nabla_{x'}\LL_i(x+z)$ with probability 1 given that $\LL_i$ is differentiable almost everywhere. Over the course of running PISGD, there are a countable number of partial derivatives to be approximated, hence $\widetilde{\nabla} F(w_l^k,\xi_l^k)$ equals the gradient for $k\in[1,...,K]$ and $l\in [1,...,S]$ with probability 1 using the chain rule. In our implementation of backpropagation, the following formulas for computing an approximate gradient were used. All arguments have been omitted to make the formulas simpler.
\begin{alignat}{6}
	\widetilde{\nabla}_{b^3_j}\LL_i=&\frac{\partial\LL_i}{\partial z^3_j}\label{bp}\\
	=&(a^3_j-y^i_j)\nonumber\\	
	\widetilde{\nabla}_{W^3_{jk}}\LL_i=&\frac{\partial\LL_i}{\partial z^3_j}\frac{\partial z^3_j}{\partial H_{jk}}\frac{\partial H_{jk}}{\partial W_{jk}^3}\nonumber\\
	=&(a^3_j-y^i_j)\alpha_k^2\frac{\partial H_{jk}}{\partial W_{jk}^3}\nonumber\\
	\widetilde{\nabla}_{b^2_j}\LL_i=&\sum_{h=1}^{N_3}\frac{\partial\LL_i}{\partial z^3_h}\frac{\partial z^3_h}{\partial \alpha_j^2}\frac{\partial \alpha_j^2}{\partial z^2_j}\nonumber\\
	=&\sum_{h=1}^{N_3}(a^3_h-y^i_h)H_{hj}\frac{\partial \alpha_j^2}{\partial z^2_j}\nonumber\\
	\widetilde{\nabla}_{W^2_{jk}}\LL_i=&\sum_{h=1}^{N_3}\frac{\partial\LL_i}{\partial z^3_h}\frac{\partial z^3_h}{\partial \alpha_j^2}\frac{\partial \alpha_j^2}{\partial z^2_j}\frac{\partial z_j^2}{\partial W^2_{jk}}\nonumber\\
	=&\sum_{h=1}^{N_3}(a^3_h-y^i_h)H_{hj}\frac{\partial \alpha_j^2}{\partial z^2_j}v^i_k\nonumber
\end{alignat}	

The proof of the following property is in the appendix.  
\begin{property}\label{backprop}
The approximate stochastic gradient $\widetilde{\nabla}\LL_i$, computed using the formulas \eqref{bp}, equals the gradient of $\LL_i$ with probability 1.
\end{property}

\pgfplotsset{every axis title/.append style={at={(0.5,0.95)}}}
\pgfplotsset{every tick label/.append style={font=\footnotesize}}
\begin{figure*}[t]
	\hspace*{-0.3cm}
	\begin{tikzpicture}
	\begin{groupplot}[group style={group name=myplot,group size= 2 by 1,
		horizontal sep=1.6cm},height=8.25cm,width=8.25cm]
	\nextgroupplot[label style={font=\normalsize},xlabel=iteration, x label style={at={(axis 
			description cs:0.5,0.02)}}, ylabel=$\text{log}(f(w))$,
	y label style={at={(axis description cs:0.01,.5)}},xmin=0,xmax=62500,ymax=0.1,ymin=-3]
	\addplot[line width=1.5pt,draw=orange, dashed]
	table[x=x,y=y]{logSGDdouble.dat};\label{plots:plot2}
	\addplot[line width=1.5pt,draw=blue]
	table[x=x,y=y]{logPISGDdouble.dat};\label{plots:plot1}
	\addplot[line width=1.5pt,draw=teal, dashed]
	table[x=x,y=y]{logSGD.dat};\label{plots:plot4}
	\addplot[line width=1.5pt,draw=magenta]
	table[x=x,y=y]{logPISGD.dat};\label{plots:plot3}	
	\addplot[line width=1.5pt,draw=cyan, dashed]
	table[x=x,y=y]{logSGDhalf.dat};\label{plots:plot6}
	\addplot[line width=1.5pt,draw=red]
	table[x=x,y=y]{logPISGDhalf.dat};\label{plots:plot5}		
	\nextgroupplot[label style={font=\normalsize},xlabel=iteration, x label style={at={(axis 
			description cs:0.5,0.02)}},ylabel=$f(w)$,y label style={at={(axis description 
			cs:0.03,.5)}},,xmin=60000,xmax=62501,ymax=0.084,ymin=0.05]
	\addplot[line width=1.5pt,draw=orange, dashed]
	table[x=x,y=y]{lastSGDdouble.dat};\label{plots:plot2}
	\addplot[line width=1.5pt,draw=blue]
	table[x=x,y=y]{lastPISGDdouble.dat};\label{plots:plot1}	
	\addplot[line width=1.5pt,draw=teal, dashed]
	table[x=x,y=y]{lastSGD.dat};\label{plots:plot4}
	\addplot[line width=1.5pt,draw=magenta]
	table[x=x,y=y]{lastPISGD.dat};\label{plots:plot3}	
	\addplot[line width=1.5pt,draw=cyan, dashed]
	table[x=x,y=y]{lastSGDhalf.dat};\label{plots:plot6}
	\addplot[line width=1.5pt,draw=red]
	table[x=x,y=y]{lastPISGDhalf.dat};\label{plots:plot5}		
	\end{groupplot}
	\path (myplot c1r1.south west|-current bounding box.south)--
	coordinate(legendpos)
	(myplot c2r1.south east|-current bounding box.south);
	\matrix[matrix of nodes,anchor=south,draw,inner sep=0.2em,draw,
	column 2/.style={nodes={font=\footnotesize}},
	column 4/.style={nodes={font=\footnotesize}},
	column 6/.style={nodes={font=\footnotesize}},
	column 8/.style={nodes={font=\footnotesize}}]at([yshift=-1.5cm]legendpos)
	{\ref{plots:plot6}&SGD ($\eta=0.005$)&[5pt]
		\ref{plots:plot4}&SGD ($\eta=0.01$)&[5pt]
		\ref{plots:plot2}&SGD ($\eta=0.02$)&[5pt]\\
		\ref{plots:plot5}&PISGD ($\eta=0.005$)&[5pt]
		\ref{plots:plot3}&PISGD ($\eta=0.01$)&[5pt]		
		\ref{plots:plot1}&PISGD ($\eta=0.02$)&[5pt]							
		\\};
	\end{tikzpicture}
	\caption{A comparison of the performance of PISGD and SGD. The left graph has all of the iterates 
		plotted on a log scale. The right graph is the last $4\%$ of the iterates on a linear scale.} 
	\label{T1}
\end{figure*}
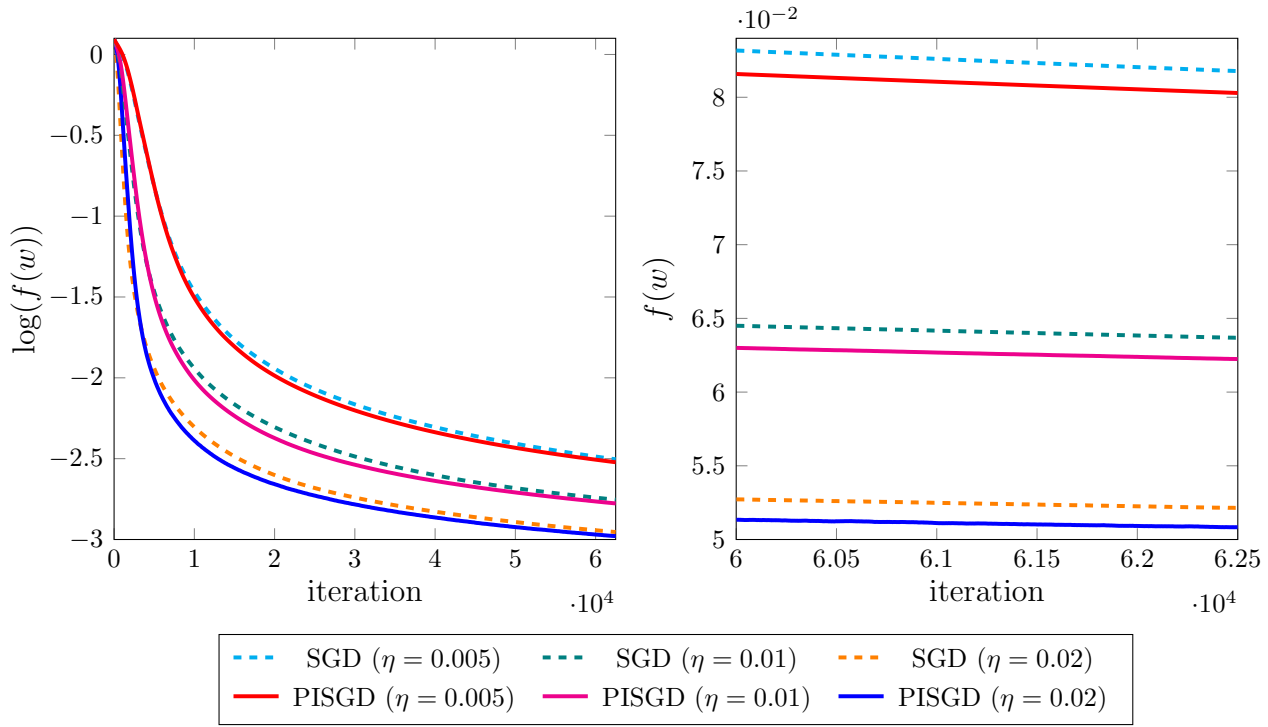

\subsection{Numerical Substantiation}
The described neural network was trained using PISGD as well as with Algorithm \ref{alg:sgd} implemented with $w^k_l=x^k$, i.e. no iterate perturbation, which 
matches how neural networks are generally trained using a mini-batch stochastic gradient descent 
algorithm, which will be referred to as SGD. We ran both algorithms under a typical implementation of SGD: 
The step size was chosen as $\eta=0.01$, which is a 
default setting when using, for example, Keras 2.3.0 \cite{chollet2015}, as well as $\eta=0.02$ and 
$\eta=0.005$. For each $\eta$, following Theorem \ref{SGDF}, $\sigma=\eta L_0\sqrt{d}$, where $d$ is 
the number of decision variables and $L_0$ is the mean value of $\{L_i\}$. The mini-batch size is 
generally chosen between 32-512 samples \cite{keskar2019}, so we chose $S=250$, which is 
roughly the average. From Theorem \ref{SGDF}, the number of iterations should satisfy 
$S=\lceil K^{1-\beta}\rceil$ for $\beta\in (0,1)$. We chose $K=62,500$, inferring a choice of 
$\beta=0.5$. 
The neural network was trained five times with each algorithm, and the function values at each iteration 
were averaged together. Both algorithms were implemented in Python 3.6 on a server running Ubuntu 16.04 with an Intel Xeon E5-2698 v4 processor. Examining Figure \ref{T1} we can see that both algorithms performed similarly, with PISGD producing a slightly lower loss function than SGD in this application. 

\section{Conclusion}
\label{con}
In this paper a new variant of stochastic gradient descent, PISGD, was developed which contains two 
forms of randomness in the step direction from sampling the stochastic function's 
approximate gradient at randomly perturbed iterates. Using this methodology, non-asymptotic convergence to an expected  
$(\epsilon_1,\epsilon_2)$-stationary point was proven for minimizing stochastic Lipschitz continuous loss functions. From this result, the computational complexities for finding an expected  
$(\epsilon_1,\epsilon_2)$-stationary point and an $(\epsilon_1,\epsilon_2)$-stationary point with high probability were given, as well as a method to obtain a Clarke stationary point almost surely.\\

{\bf Acknowledgments} The research of the second author is supported in part by JSPS KAKENHI Grant Number 19H04069.

\appendix
{\normalsize

\section{Proofs and Auxiliary Results}

\subsection{Section \ref{int}}

\begin{prop}
	\label{LCGprop}	
	A bounded function $f(w)$ such that $|f(w)|\leq R$ for all $w\in\RR^d$, with a Lipschitz 
	continuous gradient with parameter $L_1$, is Lipschitz continuous with parameter 
	$L_0=2R+\frac{L_1}{2}d$.		
\end{prop}

\begin{proof}
	A function has a Lipschitz continuous gradient if there exists a constant $L_1$ such that for all 
	$x,w\in \RR^d$, $\norm{\nabla f(x)-\nabla f(w)}_2\leq L_1\norm{x-w}_2$, which is equivalent to 
	(see \cite[Lemma 1.2.3]{nester2004})
	\begin{alignat}{6}
	|f(x)-f(w)-\langle \nabla f(w),x-w\rangle|&\leq&&\frac{L_1}{2}\norm{x-w}^2_2.\label{eq:smoov2}
	\end{alignat}
	
	By the mean value theorem, if a differentiable function has a bounded gradient such 
	that $\norm{\nabla f(w)}_2\leq L_0$ for all $w\in \RR^n$, then it is Lipschitz continuous with 
	parameter $L_0$. Using \eqref{eq:smoov2} with $x=w-y$ for any $y\in \RR^d$,	
	$$f(w-y)-f(w)+\langle \nabla f(w),y\rangle\leq\frac{L_1}{2}\norm{y}^2_2.$$	
	Taking $y_j=\sgn(\nabla_j f(w))$ for $j=1,...,d$, and using the boundedness of $f(w)$,  
	$$\norm{\nabla f(w)}_2\leq \norm{\nabla f(w)}_1\leq 2R+\frac{L_1}{2}d.$$
\end{proof}

\subsection{Section \ref{pre}}

\begin{customproperty}{\cite[Lemma 1]{bianchi2020}}{\ref{AAder}}
	The stochastic function $F(w,\xi)$ is differentiable in $w$ almost everywhere on the product measure space 
		${(\RR^{d+p},\B_{\RR^{d+p}},m^d\times P_{\xi})}$.	
\end{customproperty}

\begin{proof}
	Let $D\subseteq\{(w,\xi):w\in\RR^d,\text{ }\xi\in\Xi\}=:\overline{\Xi}$ be the set of points where $F(w,\xi)$ is differentiable in $w$ within the set $\overline{\Xi}$ where it is Lipschitz continuous in $w$. For $F(w,\xi)$ to be differentiable at a point $(w,\xi)$, there exists a unique $g\in\RR^d$ such that for any $\omega>0$, there exists a $\delta>0$ such that for all $h\in \RR^d$ where $0<||h||_2<\delta$, it holds that 
	\begin{alignat}{6}
	\frac{|F(w+h,\xi)-F(w,\xi)-\langle g,h\rangle |}{||h||_2}<\omega.\nonumber
	\end{alignat} 	
	For simplicity let $H(w,\xi,h,g):=\frac{|F(w+h,\xi)-F(w,\xi)-\langle g,h\rangle |}{||h||_2}$. The set $D$ can be represented as 
	\begin{alignat}{6}
	\bigcup_{g\in\RR^d}\bigcap_{\omega\in\QQ_{>0}}
	\bigcup_{\delta\in\QQ_{>0}}\bigcap_{\substack{0<||h||_2<\delta\\h\in\QQ^d}}	
	\left\{(w,\xi)\in \overline{\Xi}:H(w,\xi,h,g)<\omega\right\},\nonumber
	\end{alignat} 	
	where $h$ can be restricted to be over $\QQ^d$ as $H(w,\xi,h,g)$ is continuous in $h$ when $||h||_2>0$ and $\QQ^d$ is dense in $\RR^d$. We want to prove that the set $\hat{D}$ defined as   
	\begin{alignat}{6}
	\bigcap_{\omega\in\QQ_{>0}}
	\bigcup_{(g,\delta)\in\QQ^d\times\QQ_{>0}}\bigcap_{\substack{0<||h||_2<\delta\\h\in\QQ^d}}	
	\left\{(w,\xi)\in \overline{\Xi}:H(w,\xi,h,g)<\omega\right\},\nonumber
	\end{alignat} 
	is equal to $D$, proving that $D$ is an element of $\B_{\RR^{d+p}}$.\\

	For an element $(w',\xi')\in D$ with $g'$ being the gradient at $(w',\xi')$, for any $\omega>0$, take $\delta(\frac{\omega}{2})>0$ such that
	\begin{alignat}{6}
	(w',\xi')\in\bigcap\limits_{\substack{0<||h||_2<\delta(\frac{\omega}{2})\\h\in\QQ^d}}	\left\{(w,\xi)\in \overline{\Xi}:H(w,\xi,h,g')<\frac{\omega}{2}\right\},\nonumber
	\end{alignat} 
	and take $g\in \QQ^d$ such that $||g'-g||_2<\frac{\omega}{2}$. It follows that
	\begin{alignat}{6}
	&&H(w',\xi',h,g')&<\frac{\omega}{2}\nonumber\\
	\Longrightarrow&&\frac{|F(w'+h,\xi')-F(w',\xi')-\langle g,h\rangle-
	\langle g'-g,h\rangle|}{||h||_2}&<\frac{\omega}{2}\nonumber\\
	\Longrightarrow&&H(w',\xi',h,g)-\frac{|\langle g'-g,h\rangle|}{||h||_2}&<\frac{\omega}{2}\nonumber\\
	\Longrightarrow&&H(w',\xi',h,g)&<\omega\nonumber
	\end{alignat} 
	when $0<||h||_2<\delta(\frac{\omega}{2})$, using the reverse triangle inequality for the third inequality, proving that $(w',\xi')\in \hat{D}$.\\ 
	
	Considering now an element $(w',\xi')\in \hat{D}$, let $\{\omega_i\}\subset \QQ_{>0}$ be a non-increasing sequence approaching zero in the limit, with $\{g_i\}\subset \QQ^d$, and let $\{\delta_i\}\subset\QQ_{>0}$ be a non-increasing sequence such that for all $i\in\NN$, $H(w',\xi',h,g_i)<\omega_i$ when $0<||h||_2<\delta_i$. The sequence $\{g_i\}$ is bounded as 
	\begin{alignat}{6}
	&&H(w',\xi',h,g_i)&<\omega_i\nonumber\\
	\Longrightarrow&&\frac{|F(w'+h,\xi')-F(w',\xi')-\langle g_i,h\rangle|}{||h||_2}&<\omega_i\nonumber\\
	\Longrightarrow&&\frac{|\langle g_i,h\rangle|}{||h||_2}-\frac{|F(w'+h,\xi')-F(w',\xi')|}{||h||_2}&<\omega_i\nonumber\\
	\Longrightarrow&&\frac{|\langle g_i,h\rangle|}{||h||_2}&<\omega_i+C(\xi')\label{gbound}	
	\end{alignat} 
	for all $0<||h||_2<\delta_i$, using again the reverse triangle inequality and the Lipschitz continuity of $F(w,\xi')$. Taking $h=\delta'_i\frac{g_i}{||g_i||_2}$ for any $\delta'_i<\delta_i$ in \eqref{gbound}, 
	\begin{alignat}{6}
	||g_i||_2&<\omega_i+C(\xi')\leq\omega_1+C(\xi').\nonumber
	\end{alignat} 
	
	Given that the sequence $\{g_i\}$ is bounded, it contains at least one accumulation point $g'$. There then exists a subsequence $\{i_j\}\subset \NN$ such that for any $\omega\in \QQ_{>0}$, there exists a $J\in\NN$ such that for $j>J$, $\omega_{i_j}<\frac{\omega}{2}$ and $||g_{i_j}-g'||_2<\frac{\omega}{2}$, from which it holds that  
	$H(w',\xi',h,g')\leq H(w',\xi',h,g_{i_j})+||g'-g_{i_j}||_2<\omega$ when $0<||h||_2<\delta_{i_j}$, proving $g'$ is the gradient of $F(w,\xi)$ at $(w',\xi')$ and $(w',\xi')\in D$.\\   
	
	We now want to establish that $F(w,\xi)$ is differentiable almost everywhere in $w$. Let 
	$\ind_{D^c}(w,\xi)$ be the indicator function of the complement of $D$. The set $D^c$ is the set of points $(w,\xi)$ where $F(w,\xi)$ is not differentiable or not Lipschitz continuous in $w$. Showing that $D^c$ is a null set is then sufficient. Given that the function $\ind_{D^c}(w,\xi)\in L^+(\RR^d\times\RR^p)$, and $m^d$ 
	and 
	$P$ are $\sigma$-finite, the measure of $D^c$ can be computed by the iterated integral
	\begin{alignat}{6}
	\EE_{\xi}\left[\medint\int_{w\in\RR^d}\ind_{D^c}(w,\xi)dw\right]\nonumber
	\end{alignat} 
	by Tonelli's theorem. Let $\overline{\xi}\in\RR^p$ be chosen such that $F(w,\overline{\xi})$ is 
	Lipschitz continuous in $w$. By Rademacher's theorem, $F(w,\overline{\xi})$ is differentiable in 
	$w$ 
	almost everywhere, which implies that  
	$\medint\int_{w\in\RR^d}\ind_{D^c}(w,\overline{\xi})dw=0$. As this holds for almost every $\xi$, $\EE_{\xi}[\medint\int_{w\in\RR^d}\ind_{D^c}(w,\xi)dw]=0$ 
	\cite[Proposition 2.16]{folland1999}. 
\end{proof}

\begin{ex}\label{appgrad}
Let $e_j$ for $j=1,...,d$ denote the standard basis of $\RR^d$. For $i\in\NN$, let
\begin{alignat}{6}
&h^i_j(w,\xi)&&:=i(F(w+i^{-1}e_j,\xi)-F(w,\xi))\nonumber
\end{alignat} 
define a sequence $\{h^i_j(w,\xi)\}_{i\in\NN}$ of real-valued Borel measurable functions.
It holds that $h^+_j(w,\xi):=\limsup\limits_{i\rightarrow \infty} 
h^i_j(w,\xi)$ and $h^-_j(w,\xi):=\liminf\limits_{i\rightarrow \infty} 
h^i_j(w,\xi)$ are extended real-valued Borel measurable functions \cite[Lemma 2.9]{bartle1995}. For $\zeta\in[0,1]$ and $a\in \RR$, a family of candidate approximate gradients can be defined as having components  
\begin{alignat}{6}\label{famapp}
\widetilde{\nabla} F_j(w,\xi)=\begin{cases}
\zeta h^+_j(w,\xi)+(1-\zeta)h^-_j(w,\xi)&\text{ if } \{h^+_j(w,\xi), h^-_j(w,\xi)\}\in\RR\\
a & \text{ otherwise.}  
\end{cases}
\end{alignat} 
The function $\widetilde{\nabla} F(w,\xi)$ will equal $\nabla F(w,\xi)$ 
wherever it 
exists. The set 
$$A=\{(w,\xi): |h^+_j(w,\xi)|<\infty\}\cap\{(w,\xi):|h^-_j(w,\xi)|<\infty\}$$ is measurable given that for an extended real-valued measurable function $h$, the set $\{|h|=\infty\}$ is measurable \cite[Page 11]{bartle1995}. 
Let $\ind_A(w,\xi)$ and $\ind_{A^c}(w,\xi)$ denote the indicator functions of $A$ and its complement. The product of extended real-valued functions is measurable \cite[Page 12-13]{bartle1995}, implying that 
$\zeta h^+_j(w,\xi)\ind_{A}(w,\xi)$, $(1-\zeta)h^-_j(w,\xi)\ind_{A}(w,\xi)$, and 
$a\ind_{A^c}(w,\xi)$ are all measurable. 
Given that all three functions are real-valued,\footnote{Using the standard convention that $0\cdot\infty=0$.} their sum is measurable, implying the measurability of \eqref{famapp}.
\end{ex}

\begin{customproperty}{\ref{expg}}
	For almost every $w\in\RR^d$   	
	\begin{alignat}{6}
	&\widetilde{\nabla} f(w)=\EE_{\xi}[\widetilde{\nabla} F(w,\xi)],\nonumber
	\end{alignat}	
	where $\widetilde{\nabla} f(w)$ and $\widetilde{\nabla} F(w,\xi)$ are approximate gradients of 
	$f(w)$ and $F(w,\xi)$, 
	respectively. 
\end{customproperty}

\begin{proof}
	Following the proof of Property \ref{AAder}, let $D^c\subset\{(w,\xi):w\in\RR^d,\text{ }\xi\in\RR^p\}$ be the same Borel measurable set containing the points where $F(w,\xi)$ is not differentiable or not Lipschitz continuous in $w$. By Tonelli's theorem, it was established that
	\begin{alignat}{6}\label{eq:52}
	\medint\int_{w\in\RR^d}\EE_{\xi}[\ind_{D^c}(w,\xi)]dw=0.
	\end{alignat} 
	The function $G(w):=\EE_{\xi}[\ind_{D^c}(w,\xi)]$ is measurable in $(\RR^d,\B_{\RR^{d}})$, hence the set 
	$D_w:=\{w\in \RR^d: G(w)=0\}$ is measurable with full measure by \eqref{eq:52}. As in Example \ref{appgrad}, let $h^i_j(w,\xi):=i(F(w+i^{-1}e_j,\xi)-F(w,\xi))$ for $i\in\NN$.
	For all $w\in D_w$, $\lim\limits_{i\rightarrow \infty} h^i_j(w,\xi)=\widetilde{\nabla} F_j(w,\xi)$
	for almost all $\xi$, by the assumption that $\widetilde{\nabla} F(w,\xi)=\nabla F(w,\xi)$ almost everywhere $F(w,\xi)$ is differentiable. Given the Lipschitz continuity condition of $F(w,\xi)$,	for all $i\in\NN$,
	\begin{alignat}{6}
	&|h^i_j(w,\xi)|&&\leq iC(\xi)|i^{-1}|=C(\xi)\nonumber
	\end{alignat}	
	for almost all $\xi$. Given that $C(\xi)\in L^1(P_{\xi})$, the 
	dominated convergence theorem can be applied for all $w\in D_w$. It follows that  
	\begin{alignat}{6}
	&\EE_{\xi}[\widetilde{\nabla} F_j(w,\xi)]&&\stackrel{a.e.}{=}\lim\limits_{i\rightarrow \infty} 
	\EE_{\xi}\left[h_j(w,\xi,i)\right]\nonumber\\
	&&&=\lim\limits_{i\rightarrow \infty} i(f(w+i^{-1}e_j)-f(w))\nonumber\\
	&&&\stackrel{a.e.}{=}\nabla f_j(w)\nonumber\\
	&&&\stackrel{a.e.}{=}\widetilde{\nabla} f_j(w),\nonumber
	\end{alignat}	
	where the first equality holds for all $w\in D_w$, the third equality holds for almost all $w$ due to Rademacher's theorem, and the last equality holds almost everywhere by assumption.
\end{proof}

\begin{customproperty}{\ref{SetValMeas}}
	For any $\epsilon\geq 0$, $\dist(0,\partial_{\epsilon}f(w))$ is a Borel measurable function in $w\in\RR^d$.
\end{customproperty}

\begin{proof}
The set valued function \eqref{clesub} is outer semicontinuous \cite[Proposition 2.7]{goldstein1977}. The function $\dist(0,\partial_{\epsilon}f(w))$ is then lower semicontinuous \cite[Proposition 5.11 (a)]{rockafellar2009}, hence Borel measurable. 
\end{proof}

\subsection{Section \ref{ssd}}

\begin{customlemma}{\ref{ineqsmooth}}	
For $\{x,x',z\}\in \RR^d$, let $w=x+z$ and $w'=x'+z$. For a Lipschitz continuous function 
$f(\cdot)$ 
with approximate gradient $\widetilde{\nabla} f(\cdot)$, and any $x,x'\in \RR^d$,
\begin{alignat}{6}	
&f(w)-f(w')-\langle \widetilde{\nabla} f(w'),x-x'\rangle=\medint\int_0^1\langle\widetilde{\nabla} 
f(w'+v(x-x'))-\widetilde{\nabla} 
f(w'),x-x'\rangle dv\nonumber
\end{alignat}	
holds for almost all $z\in\RR^n$. 	
\end{customlemma}

\begin{proof}
	Throughout the proof let $\{x,x'\}\in \RR^d$ be fixed. Consider the function in 
	$v\in[0,1]$,	
	\begin{alignat}{6}
	&&\hat{f}_z(v)&=f(x'+z+v(x-x')),\nonumber
	\end{alignat}
	for any $z\in \RR^d$. Where it exists, 
	\begin{alignat}{6}
	&\hat{f}_z'(v)&=\lim\limits_{h\rightarrow 
		0}&\frac{f(x'+z+(v+h)(x-x'))-f(x'+z+v(x-x'))}{h}\nonumber\\
	&&=\lim\limits_{h\rightarrow 0}&\frac{f(x'+z+v(x-x')+h(x-x'))-f(x'+z+v(x-x'))}{h}\nonumber
	\end{alignat}
	is equal to the directional derivative of $f(\hat{w})$ at $\hat{w}=x'+z+v(x-x')$ in the direction 
	of $(x-x')$. Let $\ind_{D^c}(\cdot)$ be the indicator function of the complement of the set where 
	$f(\cdot)$ is differentiable, which is a Borel measurable function from the continuity of $f(w)$ \cite[Page 211]{federer199}. Its composition with 
	the continuous 
	function $\hat{w}$ in $(z,v)\in \RR^d\times [0,1]$ is then as well. Similar 
	to the proof of Property \ref{AAder}, using Tonelli's theorem, the measure of 
	where $f(\hat{w})$ is not differentiable can be computed as 	
	\begin{alignat}{6}\label{multint}
	\medint\int_0^1\medint\int_{z\in\RR^d}\ind_{D^c}(x'+z+v(x-x'))dzdv.
	\end{alignat}
	For any $v\in [0,1]$, $\medint\int_{z\in\RR^d}\ind_{\overline{D}}(x'+z+v(x-x'))dw=0$ by Rademacher's theorem, 
	implying that \eqref{multint} equals $0$, and $f(\hat{w})$ is differentiable for almost all 
	$(z,v)$. It follows that for almost all $(z,v)$, the directional derivative exists, the 
	approximate gradient $\widetilde{\nabla} f(\hat{w})$ is equal to the gradient, and 
	\begin{alignat}{6}\label{dirder}
	\hat{f}_z'(v)&=&&\langle\widetilde{\nabla} f(x'+z+v(x-x')),x-x'\rangle.
	\end{alignat}
	In addition $\hat{f}_z(v)$ is Lipschitz continuous, 
	\begin{alignat}{6}
	&|\hat{f}_z(v)-\hat{f}_z(v')|&=&|f(x'+z+v(x-x'))-f(x'+z+v'(x-x'))|\nonumber\\
	&&\leq&L_0\norm{x-x'}_2|v-v'|.\nonumber
	\end{alignat}
	Choosing $z=\overline{z}$ such that \eqref{dirder} holds for almost all $v\in[0,1]$, by 
	the fundamental theorem of calculus for Lebesgue integrals, 
	\begin{alignat}{6}
	&&f(x+\overline{z})=\hat{f}_{\overline{z}}(1)=&\hat{f}_{\overline{z}}(0)+\medint\int_0^1\hat{f}_{\overline{z}}'(v)dv\nonumber\\
	&&=&f(x'+\overline{z})+\medint\int_0^1\langle\widetilde{\nabla} f(x'+\overline{z}+v(x-x')),x-x'\rangle 
	dv.\nonumber
	\end{alignat}	
	Rearranging and subtracting $\langle \widetilde{\nabla} f(x'+\overline{z}),x-x'\rangle$ from both 
	sides,
	\begin{alignat}{6}	
	&f(x+\overline{z})-f(x'+\overline{z})-\langle \widetilde{\nabla} 
	f(x'+\overline{z}),x-x'\rangle\nonumber\\
	=&\medint\int_0^1\langle\widetilde{\nabla} 
	f(x'+\overline{z}+v(x-x'))-\widetilde{\nabla} 
	f(x'+\overline{z}),x-x'\rangle dv.\nonumber
	\end{alignat}
	As for almost all $z\in \RR^d$, \eqref{dirder} holds for almost all $v\in[0,1]$,
	\begin{alignat}{6}	
	&f(w)-f(w')-\langle \widetilde{\nabla} 
	f(w'),x-x'\rangle\nonumber=\medint\int_0^1\langle\widetilde{\nabla} f(w'+v(x-x'))-\widetilde{\nabla} 
	f(w'),x-x'\rangle dv\nonumber	
\end{alignat}	
holds for almost all $z\in\RR^n$. 	
\end{proof}

\begin{customlemma}{\ref{gradbound}}	
	The norms of the approximate gradients are bounded, with $\norm{\widetilde{\nabla} f(w)}_2\leq L_0$ and 
	$\norm{\widetilde{\nabla} F(w,\xi)}_2\leq C(\xi)$ almost everywhere.
\end{customlemma}

\begin{proof}
	As $f(w)$ is differentiable almost everywhere, and $\widetilde{\nabla} f(w)$ is equal to 
	the gradient of $f(w)$ almost everywhere it is differentiable, using the directional derivative 
	and Lipschitz continuity of $f(w)$, 
	\begin{alignat}{6}
	&\norm{\widetilde{\nabla} f(w)}^2_2&&=\lim\limits_{h\rightarrow 0} 
	\frac{f(w+h\widetilde{\nabla} f(w))-f(w)}{h}\nonumber\\
	&&&\leq L_0\norm{\widetilde{\nabla} f(w)}_2\nonumber
	\end{alignat}
	holds almost everywhere. Similarly, by assumption and Property \ref{AAder}, $F(w,\xi)$ 
	is Lipschitz continuous and differentiable almost everywhere, with $\widetilde{\nabla} F(w,\xi)$ equal to the 
		gradient almost everywhere $F(w,\xi)$ is differentiable. It follows that almost everywhere,
		\begin{alignat}{6}
		&\norm{\widetilde{\nabla} F(w,\xi)}^2_2&&=\lim\limits_{h\rightarrow 0} 
		\frac{F(w+h\widetilde{\nabla} F(w,\xi),\xi)-F(w,\xi)}{h}\nonumber\\
		&&&\leq C(\xi)\norm{\widetilde{\nabla} F(w,\xi)}_2.\nonumber
		\end{alignat}
\end{proof}	

\begin{customlemma}{\ref{stobound}}	
	Let $x\in\RR^d$ and $z\in\RR^d$ be random variables, where $z$ is absolutely continuous, and $x$, $z$, and $\xi$ (as previously defined) are mutually independent. For any $S\in \ZZ_{>0}$, let  $$\overline{\nabla}F:=\frac{1}{S}\sum_{l=1}^S\widetilde{\nabla}	F(x+z_l,\xi_l),$$ 	
	where $z_l\sim P_z$ and $\xi_l\sim P_{\xi}$ for $l=1,...,S$. It holds that  
	\begin{alignat}{6}
	&&&\EE[\norm{\overline{\nabla}F}^2_2
			-\norm{\EE[\overline{\nabla}F|x]}^2_2]\nonumber\\
		&&=&\EE[\norm{\overline{\nabla}F-\EE[\overline{\nabla}F|x]}^2_2]\nonumber\\
		&&\leq&\frac{Q}{S},\nonumber
	\end{alignat}
	where $Q:=\EE[C(\xi)^2]$.
\end{customlemma}
\begin{proof}
We first show that		 
		$\EE[\norm{\overline{\nabla}F}^2_2
				-\norm{\EE[\overline{\nabla}F|x]}^2_2]=
				\EE[\norm{\overline{\nabla}F-\EE[\overline{\nabla}F|x]}^2_2]$:
	\begin{alignat}{6}
	&&&\EE[\norm{\overline{\nabla}F-\EE[\overline{\nabla}F|x]}^2_2]\nonumber\\
	&&=&\EE[\norm{\overline{\nabla}F}^2_2-2\langle \overline{\nabla}F,\EE[\overline{\nabla}F|x]\rangle+
	\norm{\EE[\overline{\nabla}F|x]}^2_2]\nonumber\\
	&&=&\EE[\norm{\overline{\nabla}F}^2_2]
	-2\EE[\langle \overline{\nabla}F,\EE[\overline{\nabla}F|x]\rangle]
	+\EE[\norm{\EE[\overline{\nabla}F|x]}^2_2]\nonumber\\
	&&=&\EE[\norm{\overline{\nabla}F}^2_2]
	-2\EE(\EE[\langle \overline{\nabla}F,\EE[\overline{\nabla}F|x]\rangle|x])
	+\EE[\norm{\EE[\overline{\nabla}F|x]}^2_2]\nonumber\\
	&&=&\EE[\norm{\overline{\nabla}F}^2_2]
	-2\EE[\langle\EE[\overline{\nabla}F|x],\EE[\overline{\nabla}F|x]\rangle]
	+\EE[\norm{\EE[\overline{\nabla}F|x]}^2_2]\nonumber\\
	&&=&\EE[\norm{\overline{\nabla}F}^2_2]
	-2\EE[\norm{\EE[\overline{\nabla}F|x]}^2_2]
	+\EE[\norm{\EE[\overline{\nabla}F|x]}^2_2]\nonumber\\
	&&=&\EE[\norm{\overline{\nabla}F}^2_2]
	-\EE[\norm{\EE[\overline{\nabla}F|x]}^2_2].\nonumber
	\end{alignat}
			
	Let $w_l:=x+z_l$ for $l=1,...,S$. Analyzing now $\EE[\norm{\overline{\nabla}F-\EE[\overline{\nabla}F|x]}^2_2]$,
	\begin{alignat}{6}	
	&&&\EE[\norm{\overline{\nabla}F-\EE[\overline{\nabla}F|x]}^2_2]\nonumber\\
	&&=&\EE[\sum_{j=1}^d(\overline{\nabla}_jF-\EE[\overline{\nabla}_jF|x])^2]\nonumber\\
	&&=&\EE[\sum_{j=1}^d(\frac{1}{S}\sum\limits_{l=1}^S(\widetilde{\nabla}_j 
	F(w_l,\xi_l)-\EE[\overline{\nabla}_jF|x]))^2]\nonumber\\
	&&=&\frac{1}{S^2}\sum_{j=1}^d\EE[(\sum\limits_{l=1}^S(\widetilde{\nabla}_j	
	F(w_l,\xi_l)-\EE[\overline{\nabla}_jF|x]))^2]\nonumber\\
	&&=&\frac{1}{S^2}\sum_{j=1}^d\EE(\EE[(\sum\limits_{l=1}^S(\widetilde{\nabla}_j 	
	F(w_l,\xi_l)-\EE[\overline{\nabla}_jF|x]))^2|x])\label{crossterms}\\
	&&=&\frac{1}{S^2}\sum_{j=1}^d\EE(\sum\limits_{l=1}^S\EE[(\widetilde{\nabla}_j
	F(w_l,\xi_l)-\EE[\overline{\nabla}_jF|x])^2|x]),\label{varind}
	\end{alignat}
	where \eqref{varind} holds since $\widetilde{\nabla}_j 
	F(w_l,\xi_l)-\EE[\overline{\nabla}_jF|x]$ for 
	$l=1,...,S$ are conditionally independent random variables with conditional expectation of zero with respect to $x$: Considering the cross terms of $\EE[(\sum\limits_{l=1}^S(\widetilde{\nabla}_j 	
		F(w_l,\xi_l)-\EE[\overline{\nabla}_jF|x]))^2|x]$ in \eqref{crossterms} with $l\neq m$, 
	\begin{alignat}{6}
	&&&\EE[(\widetilde{\nabla}_j 	
			F(w_l,\xi_l)-\EE[\overline{\nabla}_jF|x])(\widetilde{\nabla}_j 	
					F(w_m,\xi_m)-\EE[\overline{\nabla}_jF|x])|x]\nonumber\\	
	&&=&\EE[\widetilde{\nabla}_jF(w_l,\xi_l) \widetilde{\nabla}_jF(w_m,\xi_m)|x]
	-\EE[\widetilde{\nabla}_jF(w_l,\xi_l)\EE[\overline{\nabla}_jF|x]|x]\nonumber\\
	&&-&\EE[\EE[\overline{\nabla}_jF|x]\widetilde{\nabla}_jF(w_m,\xi_m)|x]	
	+\EE[\EE[\overline{\nabla}_jF|x]^2|x]\nonumber\\
	&&=&\EE[\widetilde{\nabla}_jF(w_l,\xi_l)|x]\EE[\widetilde{\nabla}_jF(w_m,\xi_m)|x]
	-\EE[\widetilde{\nabla}_jF(w_l,\xi_l)|x]\EE[\overline{\nabla}_jF|x]\nonumber\\
	&&-&\EE[\overline{\nabla}_jF|x]\EE[\widetilde{\nabla}_jF(w_m,\xi_m)|x]	
	+\EE[\overline{\nabla}_jF|x]^2\nonumber\\		
	&&=&0.\nonumber
	\end{alignat}

	Continuing from \eqref{varind}, 	
	\begin{alignat}{6}
	&&&\frac{1}{S^2}\sum_{j=1}^d\EE(\sum\limits_{l=1}^S\EE[(\widetilde{\nabla}_j
	F(w_l,\xi_l)-\EE[\overline{\nabla}_jF|x])^2|x])\nonumber\\	
	&&=&\frac{1}{S^2}\sum_{j=1}^d\sum\limits_{l=1}^S\EE[(\widetilde{\nabla}_j 
	F(w_l,\xi_l)-\EE[\overline{\nabla}_jF|x])^2]\nonumber\\
	&&=&\frac{1}{S}\sum_{j=1}^d\EE[(\widetilde{\nabla}_j 
	F(w_l,\xi_l)-\EE[\overline{\nabla}_jF|x])^2],\label{eq:25}
	\end{alignat}
	where the last equality holds for any $l\in\{1,...,S\}$. Continuing from \eqref{eq:25},
	\begin{alignat}{6}
	&&&\frac{1}{S}\sum_{j=1}^d\EE[(\widetilde{\nabla}_j 
		F(w_l,\xi_l)-\EE[\overline{\nabla}_jF|x])^2]\nonumber\\
	&&=&\frac{1}{S}\sum_{j=1}^d\EE(\EE[(\widetilde{\nabla}_j 
					F(w_l,\xi_l)-\EE[\overline{\nabla}_jF|x])^2|x])\nonumber\\
	&&=&\frac{1}{S}\sum_{j=1}^d\EE(\EE[\widetilde{\nabla}_j 
					F(w_l,\xi_l)^2-2\widetilde{\nabla}_j 
										F(w_l,\xi_l)\EE[\overline{\nabla}_jF|x]+  \EE[\overline{\nabla}_jF|x]^2|x])\nonumber\\	
	&&=&\frac{1}{S}\sum_{j=1}^d\EE(\EE[\widetilde{\nabla}_j 
					F(w_l,\xi_l)^2|x]-2\EE[\widetilde{\nabla}_j 
										F(w_l,\xi_l)|x]\EE[\overline{\nabla}_jF|x]+  \EE[\overline{\nabla}_jF|x]^2)\nonumber\\
	&&=&\frac{1}{S}\sum_{j=1}^d\EE(\EE[\widetilde{\nabla}_j 
					F(w_l,\xi_l)^2|x]-2\EE[\overline{\nabla}_jF|x]^2+  \EE[\overline{\nabla}_jF|x]^2)\nonumber\\
	&&=&\frac{1}{S}\sum_{j=1}^d\EE(\EE[\widetilde{\nabla}_j 
					F(w_l,\xi_l)^2|x]-\EE[\overline{\nabla}_jF|x]^2)\nonumber\\
	&&=&\frac{1}{S}\sum_{j=1}^d\EE[\widetilde{\nabla}_j 
					F(w_l,\xi_l)^2]-\EE[\EE[\overline{\nabla}_jF|x]^2]\nonumber\\
	&&\leq&\frac{1}{S}\sum_{j=1}^d\EE[\widetilde{\nabla}_j 
					F(w_l,\xi_l)^2]\nonumber\\
	&&=&\frac{1}{S}\EE[\norm{\widetilde{\nabla} F(w_l,\xi_l)}^2_2]\nonumber\\
	&&\leq&\frac{Q}{S},\label{eq:22}
	\end{alignat}
	where the final inequality uses Lemma \ref{gradbound} and the 
	definition $Q:=\EE[C(\xi)^2]$: Similar to showing \eqref{tuffeq}, since $z_l$ and $\xi_l$ are independent of $x$, $\EE[\norm{\widetilde{\nabla} F(x+z_l,\xi_l)}^2_2|x]=g(x)$, where $g(y):=\EE[\norm{\widetilde{\nabla} F(y+z_l,\xi_l)}^2_2]$. By the absolute continuity of $z_l$, for all $y\in\RR^d$, $\norm{\widetilde{\nabla} F(y+z_l,\xi_l)}^2_2\leq C(\xi_l)^2$ for almost every $(z_l,\xi_l)$ from Lemma \ref{gradbound}, hence $g(y)\leq \EE[C(\xi)^2]$ for all $y\in\RR^d$, and in particular, $\EE[\norm{\widetilde{\nabla} F(w_l,\xi_l)}^2_2]=\EE[g(x)]\leq \EE[C(\xi)^2]$.
\end{proof}	
\begin{prop}
	\label{boundprop}	
	For $d\in\NN$,  
	$$\frac{\lambda(d)d!!}{(d-1)!!}\leq \sqrt{d}.$$	
\end{prop}

\begin{proof}
	For $d=1$, $\frac{\lambda(d)d!!}{(d-1)!!}=1$, and for $d=2$, 
	$\frac{\lambda(d)d!!}{(d-1)!!}=\frac{4}{\pi}<\sqrt{2}$. For $d\geq 2$, we will show that the result holds for 
	$d+1$ assuming that it holds for $d-1$, proving the result by induction.
	\begin{alignat}{6}
	&\frac{\lambda(d+1)(d+1)!!}{d!!}&&=\frac{\lambda(d-1)(d+1)(d-1)!!}{d(d-2)!!}\nonumber\\
	&&&=\frac{\lambda(d-1)(d-1)!!}{(d-2)!!}\frac{(d+1)}{d}\nonumber\\
	&&&\leq\sqrt{d-1}\frac{(d+1)}{d}\nonumber\\
	&&&=\sqrt{\frac{(d-1)(d+1)^2}{d^2}}\nonumber\\
	&&&=\sqrt{\frac{d^3+d^2-d-1}{d^2}}\nonumber\\
	&&&<\sqrt{d+1}.\nonumber					
	\end{alignat}	  
\end{proof}

\begin{customcorollary}{\ref{ssdcomp}}
	For $\beta\in(0,1)$, an expected $(\epsilon_1,\epsilon_2)$-stationary point \eqref{eps} can be 
	computed with 
	$O\left(\max\left(\epsilon_1^{\frac{\beta-2}{\beta}},\epsilon_2^{-2\frac{2-\beta}{1-\beta}}\right)\right)$
	gradient calls.
\end{customcorollary}

\begin{proof}	
	From Theorem \ref{SGDF}, $\sigma=\theta\sqrt{d}K^{-\beta}$, and requiring 
	$\sigma \leq \epsilon_1$ implies
	\begin{alignat}{6}
	\left(\frac{\theta\sqrt{d}}{\epsilon_1}\right)^{\frac{1}{\beta}}\leq K.\nonumber
	\end{alignat}	
	Taking $K_{\epsilon_1}=\left\lceil\left(\frac{\theta\sqrt{d}}{\epsilon_1}\right)^{\frac{1}{\beta}}\right\rceil<\left(\frac{\theta\sqrt{d}}{\epsilon_1}\right)^{\frac{1}{\beta}}+1$ and $S_{\epsilon_1}=\lceil K_{\epsilon_1}^{1-\beta}\rceil<K_{\epsilon_1}^{1-\beta}+1$, 
	\begin{alignat}{6}
	S_{\epsilon_1}&&<&\left(\left(\frac{\theta\sqrt{d}}{\epsilon_1}\right)^{\frac{1}{\beta}}+1\right)^{1-\beta}+1\nonumber\\
	&&<&\left(\frac{\theta\sqrt{d}}{\epsilon_1}\right)^{\frac{1-\beta}{\beta}}+2,\nonumber
	\end{alignat}
	where the second inequality follows from a general result for $a_i>0$ for $i=1,...,n$ and $\beta\in(0,1)$:
	$(\sum_{i=1}^na_i)^{1-\beta}=\frac{\sum_{i=1}^na_i}{(\sum_{i=1}^na_i)^{\beta}}<\sum_{i=1}^n\frac{a_i}{a_i^{\beta}}=\sum_{i=1}^na_i^{1-\beta}$.
	An upper bound on the total number of gradient calls required to satisfy $\epsilon_1$, considering up to $K_{\epsilon_1}-1$ iterations of PISGD is then 
	\begin{alignat}{6}
	(K_{\epsilon_1}-1)S_{\epsilon_1}&&<&\left(\frac{\theta\sqrt{d}}{\epsilon_1}\right)^{\frac{1}{\beta}}\left(
	\left(\frac{\theta\sqrt{d}}{\epsilon_1}\right)^{\frac{1-\beta}{\beta}}+2\right)\nonumber\\
	&&=&O\left(\epsilon_1^{\frac{\beta-2}{\beta}}\right).\nonumber
	\end{alignat}	
	Choosing $K$ such that	
	\begin{alignat}{6}
	\EE[\dist(0,\partial_{\sigma}f(x^R))]&&<&	
	K^{\frac{\beta-1}{2}}\sqrt{2\left(\frac{L_0}{\theta}\Delta+L_0^2\sqrt{d} 
	K^{-\beta}+Q\right)}\nonumber\\
	&&\leq &	
	K^{\frac{\beta-1}{2}}\sqrt{2\left(\frac{L_0}{\theta}\Delta+L_0^2\sqrt{d}+Q\right)}\nonumber\\
	&&\leq&\epsilon_2\nonumber
	\end{alignat}
	gives the bound 
	\begin{alignat}{6}
	\left(\frac{2}{\epsilon^2_2}(\frac{L_0}{\theta}\Delta+L_0^2\sqrt{d}+Q)\right)^{\frac{1}{1-\beta}}\leq K.\nonumber
	\end{alignat}
	Taking  $K_{\epsilon_2}=\left\lceil\left(\frac{2}{\epsilon^2_2}(\frac{L_0}{\theta}\Delta+L_0^2\sqrt{d}+Q)\right)^{\frac{1}{1-\beta}}\right\rceil$ and $S_{\epsilon_2}=\lceil K_{\epsilon_2}^{1-\beta}\rceil$, 
	\begin{alignat}{6}
	S_{\epsilon_2}&&<&\left(\left(\frac{2}{\epsilon^2_2}(\frac{L_0}{\theta}\Delta+L_0^2\sqrt{d}+Q)\right)^{\frac{1}{1-\beta}}+1\right)^{1-\beta}+1\nonumber\\
	&&<&\frac{2}{\epsilon^2_2}(\frac{L_0}{\theta}\Delta+L_0^2\sqrt{d}+Q)+2,\nonumber
	\end{alignat}
	and the number of gradient calls required to satisfy $\epsilon_2$ is bounded by
	\begin{alignat}{6}
	(K_{\epsilon_2}-1)S_{\epsilon_2}&&<&\left(\frac{2}{\epsilon^2_2}(\frac{L_0}{\theta}\Delta+L_0^2\sqrt{d}+Q)\right)^{\frac{1}{1-\beta}}\left(\frac{2}{\epsilon^2_2}(\frac{L_0}{\theta}\Delta+L_0^2\sqrt{d}+Q)+2\right)\nonumber\\
	&&=&O\left(\epsilon_2^{-2\frac{2-\beta}{1-\beta}}\right).\nonumber
	\end{alignat}
	The number of gradient calls required to satisfy both $\epsilon_1$ and $\epsilon_2$ is then
	\begin{alignat}{6}
	\max((K_{\epsilon_1}-1)S_{\epsilon_1},(K_{\epsilon_2}-1)S_{\epsilon_2})=O\left(\max\left(\epsilon_1^{\frac{\beta-2}{\beta}},\epsilon_2^{-2\frac{2-\beta}{1-\beta}}\right)\right).\nonumber
	\end{alignat}	
\end{proof}

\begin{customcorollary}{\ref{ssdcomp2}}	
Assume $\epsilon_2<L_0$, $\theta=1$, and let 
\begin{alignat}{6}
K^*&=&\max\left(\left\lfloor \frac{2}{\epsilon^2_2}\left(
L_0\Delta+Q+\sqrt{d}L_0^2\right)+1\right\rfloor,\left\lceil \frac{2\sqrt{d}}{\epsilon^2_2}\left(
\frac{L_0\Delta+Q}{\epsilon_1}+L_0^2\right)\right\rceil\right)\nonumber
\end{alignat}
and  
\begin{alignat}{6}  
\beta^*=\frac{\log(K^*\epsilon^2_2-2\sqrt{d}L_0^2)-\log(2(L_0\Delta+Q))}{\log(K^*)},\nonumber
\end{alignat}
then 
\begin{enumerate}
	\item an expected $(\epsilon_1,\epsilon_2)$-stationary point \eqref{eps} can be 
	computed with 	
	$O\left(\frac{1}{\epsilon_1\epsilon^4_2}\right)$ gradient calls, 
	\item $(K^*,\beta^*)$ has the same gradient call complexity as a minimizer of the number of gradient calls required to find an expected 
	$(\epsilon_1,\epsilon_2)$-stationary point using inequality \eqref{bound} of Theorem \ref{SGDF}, and
	\item $(K^*,\beta^*)$ minimizes the number of iterations 
	required to find an expected $(\epsilon_1,\epsilon_2)$-stationary point using inequality \eqref{bound} of Theorem \ref{SGDF}.
\end{enumerate}
\end{customcorollary}

\begin{proof}
	Based on Theorem 
	\ref{SGDF}, an optimal choice for $K\in \ZZ_{>0}$ and $\beta\in (0,1)$ can be written as the following optimization 
	problem for the minimization 
	of the number of gradient calls, 
	\begin{alignat}{6}
	\min\limits_{K, \beta}&\text{ }&&(K-1)\lceil K^{1-\beta}\rceil\label{eps2eq2}\\
	\st&&&\sqrt{d}K^{-\beta}\leq \epsilon_1\nonumber\\
	&&&K^{\frac{\beta-1}{2}}\sqrt{2\left(L_0\Delta+L_0^2\sqrt{d}
		K^{-\beta}+Q\right)}\leq \epsilon_2\nonumber\\
	&&&K\in\ZZ_{>0},\quad \beta\in(0,1),\nonumber 
	\end{alignat}
	requiring $\sigma\leq \epsilon_1$ and the right-hand side of \eqref{bound} to be less than or 
	equal to $\epsilon_2$. 	Rearranging the inequalities and adding the valid inequality $1\leq 
	K^{\beta}$ given the 
	constraints 
	on $K$ and $\beta$, \eqref{eps2eq2} can be rewritten as 
	\begin{alignat}{6}
	\min\limits_{K, \beta}&\text{ }&&(K-1)\lceil K^{1-\beta}\rceil\label{trueobj}\\
	\st&&&\max\left(1,\frac{\sqrt{d}}{ \epsilon_1}\right)\leq K^{\beta}\nonumber\\
	&&&K^{\beta}\leq \frac{K\epsilon^2_2-2L_0^2\sqrt{d}}{2(L_0\Delta+Q)}\nonumber\\
	&&&K\in\ZZ_{>0},\quad \beta\in(0,1).\nonumber 
	\end{alignat}
	It is first shown that $K^*$ is a lower bound for a feasible $K$ to problem \eqref{trueobj}.
	For the case $\epsilon_1<\sqrt{d}$, minimizing the 
	gap between $\frac{\sqrt{d}}{ \epsilon_1}$ and 
	$\frac{K\epsilon^2_2-2L_0^2\sqrt{d}}{2(L_0\Delta+Q)}$ sets $K$ equal to 
	\begin{alignat}{6}
	K^*_l&=&\left\lceil \frac{2\sqrt{d}}{\epsilon^2_2}\left(
	\frac{L_0\Delta+Q}{\epsilon_1}+L_0^2\right)\right\rceil,\nonumber
	\end{alignat}
	i.e. $K^*_l$ is the minimum $K\in\ZZ_{>0}$ such that $\frac{\sqrt{d}}{ 
		\epsilon_1}\leq\frac{K\epsilon^2_2-2L_0^2\sqrt{d}}{2(L_0\Delta+Q)}$.
	When $\epsilon_1\geq\sqrt{d}$, minimizing the gap between $1$ and 
	$\frac{K\epsilon^2_2-2L_0^2\sqrt{d}}{2(L_0\Delta+Q)}$ requires 
	\begin{alignat}{6}
	&K&\geq&\frac{2}{\epsilon^2_2}\left(L_0\Delta+Q+\sqrt{d}L_0^2\right)\nonumber\\
	&&>&2\sqrt{d},\nonumber
	\end{alignat}
	given that $\epsilon_2<L_0$. 
	A valid lower bound on $K$ equals   
	\begin{alignat}{6}
	K^*_g&=&\left\lfloor\frac{2}{\epsilon^2_2}\left(L_0\Delta+Q+\sqrt{d}L_0^2\right)+1\right\rfloor.\nonumber
	\end{alignat}
	The use of $\lfloor \cdot +1 \rfloor$ instead of $\lceil\cdot\rceil$ in this case takes into 
	account 
	of the possibility that $\frac{2}{\epsilon^2_2}\left(
	L_0\Delta+Q+L_0^2\sqrt{d}\right)\in \ZZ_{>0}$. Trying to use 
	$K=\left\lceil\frac{2}{\epsilon^2_2}\left(L_0\Delta+Q+L_0^2\sqrt{d}\right)\right\rceil$ would set 
	$\frac{K\epsilon^2_2-2L_0^2\sqrt{d}}{2(L_0\Delta+Q)}=1$, which would imply 
	$\beta=0$ as $K>2\sqrt{d}$. Comparing $K^*_l$ and $K^*_g$, for the case when $\epsilon_1<\sqrt{d}$, since
	$K^*_l>\frac{2}{\epsilon^2_2}\left(L_0\Delta+Q+\sqrt{d}L_0^2\right)$, it follows that $K^*_l\geq 
	K^*_g$, and when $\epsilon_1\geq\sqrt{d}$,
	\begin{alignat}{6}
	&K^*_g&\geq&\left\lfloor\frac{2\sqrt{d}}{\epsilon^2_2}\left(
	\frac{L_0\Delta+Q}{\epsilon_1}+L_0^2\right)+1\right\rfloor\nonumber\\
	&&\geq&K^*_l,\nonumber
	\end{alignat}
	hence $K^*$ is a valid lower bound for the number of iterations over all values of $\epsilon_1$.\\ 
	
	For any fixed $K$, the 
	objective of \eqref{trueobj} is minimized by maximizing $\beta$, so we would want to set $\beta=\beta^*_K$ such 
	that 
	$$K^{\beta^*_K}=\frac{K\epsilon^2_2-2L_0^2\sqrt{d}}{2(L_0\Delta+Q)},$$
	which equals    
	\begin{alignat}{6}  
	\beta^*_K=\frac{\log(K\epsilon^2_2-2L_0^2\sqrt{d})-\log(2(L_0\Delta+Q))}{\log(K)}.\nonumber
	\end{alignat} 
	We will show the validity of this choice of $\beta$ for all $K\geq K^*_g$, implying the validity 
	for 
	all $K\geq 
	K^*$. For any $K\geq K^*_g>2\sqrt{d}$, the division by $\log(K)$ in $\beta^*_K$ is defined.
	We now verify that $\beta^*_K\in (0,1)$ for $K\geq K^*_g$. To show that $\beta^*_K<1$, 
	isolating $\epsilon_2^2$, we 
	require 
	\begin{alignat}{6}  
	\epsilon^2_2< 2(L_0\Delta+Q)+\frac{2L_0^2\sqrt{d}}{K},\nonumber
	\end{alignat} 	
	which holds given that $Q\geq L_0^2$ by Jensen's inequality and $\epsilon_2<L_0$. The bound 
	$\beta^*_K>0$ is equivalent to  
	\begin{alignat}{6}  
	0<K\epsilon^2_2-2L_0^2\sqrt{d}-2(L_0\Delta+Q).\nonumber
	\end{alignat} 
	For $K\geq K^*_g$,
	\begin{alignat}{6}
	&&&K\epsilon^2_2-2L_0^2\sqrt{d}-2(L_0\Delta+Q)\nonumber\\
	&&>&2\left(
	L_0\Delta+Q+L_0^2\sqrt{d}\right)-2L_0^2\sqrt{d}-2(L_0\Delta+Q)\nonumber\\
	&&\geq&0,\nonumber
	\end{alignat}
	hence for all $K\geq K^*$, $\beta^*_K$ is feasible. This also proves that $K^*$ is the minimum 
	feasible value for $K$, with $\beta^*=\beta^*_{K^*}$, proving statement 3.\\
	
	We now consider the minimization of a relaxation of \eqref{trueobj}, allowing the number of samples $S\in\RR$. As statements 1 and 2 concern the computational complexity of $(K^*,\beta^*)$ we will also now assume $\epsilon_1<1$ for simplicity. 
	\begin{alignat}{6}
	\min\limits_{K, \beta}&\text{ }&&(K-1)K^{1-\beta}\label{relobj}\\
	\st&&&\frac{\sqrt{d}}{ \epsilon_1}\leq K^{\beta}\nonumber\\
	&&&K^{\beta}\leq \frac{K\epsilon^2_2-2L_0^2\sqrt{d}}{2(L_0\Delta+Q)}\nonumber\\
	&&&\beta\in(0,1),\quad K\in\ZZ_{>0},\nonumber 
	\end{alignat}	
	Plugging in $K^{\beta^*_K}=\frac{K\epsilon^2_2-2L_0^2\sqrt{d}}{2(L_0\Delta+Q)}$, the 
	optimization 
	problem becomes
	\begin{alignat}{6}
	\min\limits_{K\in\ZZ_{>0}}&\text{ }&&(K-1) 
	\frac{K2(L_0\Delta+Q)}{K\epsilon^2_2-2L_0^2\sqrt{d}}\nonumber\\
	\st&&&K\geq K^*.\nonumber
	\end{alignat}
	For simplicity 
	let $$a:=2(L_0\Delta+Q)\quad\text{and}\quad b:=2L_0^2\sqrt{d}.$$
	The problem is now
	\begin{alignat}{6}
	\min\limits_{K\in\ZZ_{>0}}&\text{ }&&\frac{aK(K-1)}{\epsilon^2_2K-b}\label{relx}\\
	\st&&&K\geq K^*.\nonumber
	\end{alignat}
	We will prove that $K^*$ is optimal for problem \eqref{relx} by showing that the derivative of the objective function with respect to $K$ is positive for $K\geq K^*$.
	\begin{alignat}{6}
	&\frac{d}{dK}\frac{aK(K-1)}{\epsilon^2_2K-b}=\frac{a(\epsilon^2_2K^2-2bK+b)}{(\epsilon^2_2K-b)^2}\nonumber
	\end{alignat}
	is non-negative for $K\geq \frac{b+\sqrt{b(b-\epsilon^2_2)}}{\epsilon^2_2}$ and positive for 
	$K\geq 
	\frac{2b}{\epsilon^2_2}$ by removing $\epsilon^2_2$ in the numerator. Written in full 
	form, the objective is increasing in $K$ for  
	\begin{alignat}{6}
	K&&\geq&\frac{4L_0^2\sqrt{d}}{\epsilon^2_2}.\nonumber
	\end{alignat}	
	Comparing this inequality with $K^*=K^*_l$ given that $\epsilon_1<1$, 	
	\begin{alignat}{6}
	K^*&>\frac{2\sqrt{d}}{\epsilon^2_2}\left(L_0\Delta+Q+L_0^2\right)\nonumber\\
	&\geq\frac{2\sqrt{d}}{\epsilon^2_2}\left(Q+L_0^2\right)\nonumber\\
	&\geq\frac{4L_0^2\sqrt{d}}{\epsilon^2_2},\nonumber
	\end{alignat}	
	using $Q\geq L_0^2$, hence over the feasible $K\geq K^*$, the objective \eqref{relx} is 
	increasing, 
	and 
	$(K^*,\beta^*)$ is an optimal solution of \eqref{relobj}.\\ 
	
	Writing $K^*=K^*_l=\left\lceil\frac{1}{\epsilon^2_2}\left(
	\frac{a\sqrt{d}}{\epsilon_1}+b\right)\right\rceil$, a bound on the optimal value of the relaxed problem \eqref{relx} gives 
	\begin{alignat}{6}
	&&&(K^*-1)\left(\frac{aK^*}{\epsilon^2_2K^*-b}\right)\nonumber\\
	&&<&\frac{1}{\epsilon^2_2}\left(
	\frac{a\sqrt{d}}{\epsilon_1}+b\right)\left(\frac{aK^*}{\epsilon^2_2K^*-b}\right)\nonumber\\
	&&\leq&\frac{1}{\epsilon^2_2}\left(
	\frac{a\sqrt{d}}{\epsilon_1}+b\right)\left(\frac{\frac{a}{\epsilon^2_2}\left(
		\frac{a\sqrt{d}}{\epsilon_1}+b\right)}{\left(
		\frac{a\sqrt{d}}{\epsilon_1}+b\right)-b}\right)\nonumber\\
	&&=&\frac{1}{\epsilon^4_2}\left(\frac{a\sqrt{d}}{\epsilon_1}+b\right)\left(a+\frac{\epsilon_1b}{\sqrt{d}}\right)\nonumber\\
	&&=&O\left(\frac{1}{\epsilon_1\epsilon^4_2}\right),\nonumber
	\end{alignat}
	where for the second inequality $\frac{aK}{\epsilon^2_2K-b}$ is decreasing in $K$, $\frac{d}{dK}\frac{aK}{\epsilon^2_2K-b}=\frac{-ab}{(\epsilon^2_2K-b)^2}$. This bound cannot be improved as 
	$(K^*-1)\left(\frac{aK^*}{\epsilon^2_2K^*-b}\right)\geq\left(\frac{1}{\epsilon^2_2}\left(
		\frac{a\sqrt{d}}{\epsilon_1}+b\right)-1\right)\left(\frac{a}{\epsilon^2_2}\right)=
	O\left(\frac{1}{\epsilon_1\epsilon^4_2}\right)$. A bound on the gradient call complexity of 
	$(K^*,\beta^*)$ for finding an expected $(\epsilon_1,\epsilon_2)$-stationary point is then  
	$(K^*-1)\left\lceil\frac{aK^*}{\epsilon^2_2K^*-b}\right\rceil<(K^*-1)\left(\frac{aK^*}{\epsilon^2_2K^*-b}+1\right)
	=O\left(\frac{1}{\epsilon_1\epsilon^4_2}\right)$, proving statement 1.\\ 
		
	Let $(\hat{K},\hat{\beta})$ be an optimal 
	solution to the original problem \eqref{trueobj} with $\epsilon_1<1$. The inequalities 
	$$(K^*-1)\frac{aK^*}{\epsilon^2_2K^*-b}\leq (\hat{K}-1)\lceil \hat{K}^{1-\hat{\beta}}\rceil\leq 
	(K^*-1)\left\lceil\frac{aK^*}{\epsilon^2_2K^*-b}\right\rceil$$
	hold since restricting $S\in\ZZ_{>0}$ cannot improve the optimal 
	objective value of \eqref{relobj}, and by the optimality of $(\hat{K},\hat{\beta})$ for problem 
	\eqref{trueobj}, respectively. This proves that using $(\hat{K},\hat{\beta})$ will result in a 
	gradient call complexity of $O\left(\frac{1}{\epsilon_1\epsilon^4_2}\right)$, proving 
	statement 2. 
\end{proof}

\begin{customcorollary}{\ref{ssdcomp3}}	
Let $c\in(0,1)$ and $\phi>1$ be arbitrary constants. For any $\gamma\in(0,1)$ and $\epsilon_2<L_0$, let PISGD be run $\R:=\lceil-\ln(c\gamma)\rceil$ times using the parameter settings of Theorem \ref{SGDF} with $\theta=1$, 
\begin{alignat}{6}
K&=&\max\left(\left\lfloor \frac{2}{(\epsilon'_2)^2}\left(
L_0\Delta+Q+\sqrt{d}L_0^2\right)+1\right\rfloor,\left\lceil \frac{2\sqrt{d}}{(\epsilon'_2)^2}\left(
\frac{L_0\Delta+Q}{\epsilon_1}+L_0^2\right)\right\rceil\right),\nonumber
\end{alignat}
and  
\begin{alignat}{6}  
\beta=\frac{\log(K(\epsilon'_2)^2-2\sqrt{d}L_0^2)-\log(2(L_0\Delta+Q))}{\log(K)},\nonumber
\end{alignat}
where $\epsilon_2'=\sqrt{\frac{\epsilon^2_2-6\psi\frac{Q}{T}}{4e}}$, $\psi=\frac{\lceil-\ln(c\gamma)\rceil}{(1-c)\gamma}$, and $T=\lceil6\phi\psi\frac{Q}{\epsilon^2_2}\rceil$, outputting candidate solutions $\overline{X}:=\{\bar{x}^1,...,\bar{x}^{\R}\}$. With $T$ samples $\{(z_1,\xi_1),...,(z_{T},\xi_{T})\}$, where $z_i\sim U(B(\sigma))$ and $\xi_i\sim P_{\xi}$ for $i=1,...,T$, let $\bar{x}^*\in\overline{X}$ be chosen such that for $\overline{\nabla} F_{T}(x):=\frac{1}{T}\sum_{t=1}^{T}\widetilde{\nabla}F(x+z_t,\xi_t)$,
\begin{alignat}{6}
\bar{x}^*\in\argmin\limits_{x\in\overline{X}}||\overline{\nabla} F_{T}(x)||_2.\nonumber
\end{alignat} 
It follows that 
\begin{enumerate}
	\item $\bar{x}^*$ is an $(\epsilon_1,\epsilon_2)$-stationary point with a probability of at least $1-\gamma$, and 
	\item the described method requires $\tilde{O}\left(\frac{1}{\epsilon_1\epsilon^4_2}+\frac{1}{\gamma\epsilon^2_2}\right)$ gradient calls.
\end{enumerate}
\end{customcorollary}
\begin{proof}
Let $\overline{\nabla} f(x):=\EE[\widetilde{\nabla} f(x+z)|x]$ for $z\sim U(B(\sigma))$. Following \cite[Eq. 2.28]{ghadimi2013} for the first inequality,\footnote{The derivation of this bound is independent of how $\overline{\nabla} f(\cdot)$ and $\overline{\nabla} F_T(\cdot)$ are defined.} 
\begin{alignat}{6}
||\overline{\nabla} f(\bar{x}^*)||^2_2\leq &4\min_{i=1,...,\R}||\overline{\nabla} f(\bar{x}^i)||^2_2+
4\max_{i=1,...,\R}||\overline{\nabla} F_T(\bar{x}^i)-\overline{\nabla} f(\bar{x}^i)||^2_2+2||\overline{\nabla} F_T(\bar{x}^*)-\overline{\nabla} f(\bar{x}^*)||^2_2\nonumber\\
\leq &4\min_{i=1,...,\R}||\overline{\nabla} f(\bar{x}^i)||^2_2+
6\max_{i=1,...,\R}||\overline{\nabla} F_T(\bar{x}^i)-\overline{\nabla} f(\bar{x}^i)||^2_2,\label{probbound}
\end{alignat}
where the second inequality holds since $2||\overline{\nabla} F_T(\bar{x}^*)-\overline{\nabla} f(\bar{x}^*)||^2_2\leq2\max_{i=1,...,\R}$ $||\overline{\nabla} F_T(\bar{x}^i)-\overline{\nabla} f(\bar{x}^i)||^2_2$. We will compute an upper bound in probability of the left-hand side of \eqref{probbound} using the two terms of the right-hand side. Let $B$ equal the right-hand side of \eqref{bound},
\begin{alignat}{6}
B^2=2K^{\beta-1}\left(L_0\Delta+L_0^2\sqrt{d} 
	K^{-\beta}+Q\right).\label{betasq}
\end{alignat}
Since $\overline{\nabla} f(x)=\EE[\frac{1}{S'}\sum_{l=1}^{S'}\widetilde{\nabla} f(x+z_l)|x]$ for any number of samples $S'\in \ZZ_{>0}$ of $z$, using inequalities \eqref{frhs} and \eqref{hpineq} of the proof of Theorem \ref{SGDF}, for all $i\in\{1,...,\R\}$,
\begin{alignat}{6}
\EE[||\overline{\nabla} f(\bar{x}^i)||^2_2]\leq B^2.\nonumber
\end{alignat}
From Markov's inequality, 
\begin{alignat}{6}
\PP(4\min_{i=1,...,\R}||\overline{\nabla} f(\bar{x}^i)||^2_2\geq 4e B^2)=\Pi_{i=1}^{\R}\PP(4||\overline{\nabla} f(\bar{x}^i)||^2_2\geq4e B^2)\leq e^{-\R}.\nonumber
\end{alignat}
Given that also $\overline{\nabla} f(x)=\EE[\frac{1}{T'}\sum_{l=1}^{T'}\widetilde{\nabla} F(x+z_l,\xi_l)|x]$ for any number of samples $T'\in \ZZ_{>0}$ of $z$ and $\xi$, $\EE[||\overline{\nabla} F_T(\bar{x}^i)-\overline{\nabla} f(\bar{x}^i)||^2_2]\leq \frac{Q}{T}$ from Lemma \ref{stobound}. For $\psi>0$, it holds that
\begin{alignat}{6}
&&&\PP\bigg(6\max_{i=1,...,\R}||\overline{\nabla} F_T(\bar{x}^i)-\overline{\nabla} f(\bar{x}^i)||^2_2\geq 6\psi \frac{Q}{T}\bigg)\nonumber\\
=&&&\PP\bigg(\bigcup_{i=1}^\R \bigg\{6||\overline{\nabla} F_T(\bar{x}^i)-\overline{\nabla} f(\bar{x}^i)||^2_2\geq 6\psi \frac{Q}{T}\bigg\}\bigg)\nonumber\\
\leq&&&\sum_{i=1}^{\R}\PP(6||\overline{\nabla} F_T(\bar{x}^i)-\overline{\nabla} f(\bar{x}^i)||^2_2\geq 6\psi \frac{Q}{T})\nonumber\\
\leq&&&\sum_{i=1}^{\R}\frac{1}{\psi}=\frac{\R}{\psi},\nonumber
\end{alignat}
using Boole's and Markov's inequalities for the first and second inequalities, respectively. 
Using the fact that $\overline{\nabla} f(\bar{x}^*)\in \partial_{\sigma}f(\bar{x}^*)$ for the first inequality,
\begin{alignat}{6}
&&&\PP\bigg(\dist(0,\partial_{\sigma}f(\bar{x}^*))^2\geq 4e B^2+6\psi\frac{Q}{T}\bigg)\nonumber\\
&\leq&&\PP\bigg(||\overline{\nabla} f(\bar{x}^*)||^2_2\geq 4e B^2+6\psi\frac{Q}{T}\bigg)\nonumber\\
&\leq&&\PP\bigg(
4\min_{i=1,...,\R}||\overline{\nabla} f(\bar{x}^i)||^2_2+
6\max_{i=1,...,\R}||\overline{\nabla} F_T(\bar{x}^i)-\overline{\nabla} f(\bar{x}^i)||^2_2
\geq 4e B^2+6\psi\frac{Q}{T}\bigg)\nonumber\\
&\leq&&\PP\bigg(\{4\min_{i=1,...,\R}||\overline{\nabla} f(\bar{x}^i)||^2_2\geq 4e B^2\}\cup\{6\max_{i=1,...,\R}||\overline{\nabla} F_T(\bar{x}^i)-\overline{\nabla} f(\bar{x}^i)||^2_2\geq6\psi\frac{Q}{T}\}\bigg)\nonumber\\
&\leq&& e^{-\R}+\frac{\R}{\psi},\label{hiprobineq}
\end{alignat}
where the second inequality uses inequality \eqref{probbound}, and the third inequality holds given that the event of the left-hand side is a subset of the right-hand side: considering the contraposition, if 
\begin{alignat}{6}
&&&\{4\min_{i=1,...,\R}||\overline{\nabla} f(\bar{x}^i)||^2_2< 4e B^2\}\cap\{6\max_{i=1,...,\R}||\overline{\nabla} F_T(\bar{x}^i)-\overline{\nabla} f(\bar{x}^i)||^2_2<6\psi\frac{Q}{T}\}\nonumber
\end{alignat}
occurs then 
\begin{alignat}{6}
&&&
4\min_{i=1,...,\R}||\overline{\nabla} f(\bar{x}^i)||^2_2+
6\max_{i=1,...,\R}||\overline{\nabla} F_T(\bar{x}^i)-\overline{\nabla} f(\bar{x}^i)||^2_2< 4e B^2+6\psi\frac{Q}{T}.\nonumber
\end{alignat}
The final inequality holds using Boole's inequality, as the right-hand side is the sum of the derived upper bounds of the probabilities of the two events of the union.\\  

The total number of gradient calls required for computing 
$\overline{X}$ and then $\overline{\nabla} F_{T}(x)$ for all $x\in \overline{X}$ to find $\bar{x}^*$ is equal to $\R((K-1)\lceil K^{1-\beta}\rceil+T)$. We can write its minimization requiring that $\PP(\dist(0,\partial_{\epsilon_1}f(\bar{x}^*))> \epsilon_2)\leq \gamma$ using \eqref{hiprobineq} with \eqref{betasq} similarly to how \eqref{eps2eq2} was derived, as
	\begin{alignat}{6}
	\min\limits_{\substack{K, \beta,\\\R,T,\psi}}&\text{ }&&\R((K-1)\lceil K^{1-\beta}\rceil+T)\nonumber\\
	\st&&&\sqrt{d}K^{-\beta}\leq \epsilon_1\nonumber\\
	&&&8e K^{\beta-1}\left(L_0\Delta+L_0^2\sqrt{d} 
		K^{-\beta}+Q\right)+6\psi\frac{Q}{T}\leq \epsilon^2_2\label{eps2eq}\\
	&&&e^{-\R}+\frac{\R}{\psi}\leq \gamma\nonumber\\
	&&& K,\R,T\in\ZZ_{>0},\quad\beta\in(0,1),\quad \psi>0.\nonumber 
	\end{alignat}
Inequality \eqref{eps2eq} can be rewritten as
	\begin{alignat}{6}
	&&&2 K^{\beta-1}\left(L_0\Delta+L_0^2\sqrt{d} 
		K^{-\beta}+Q\right)\leq \frac{\epsilon^2_2-6\psi\frac{Q}{T}}{4e}.\nonumber
	\end{alignat}

We apply the choice of $K$ and $\beta$ from Corollary \ref{ssdcomp2} for finding an expected $(\epsilon_1,\epsilon_2')$-stationary point, where $\epsilon_2'=\sqrt{\frac{\epsilon^2_2-6\psi\frac{Q}{T}}{4e}}$:
\begin{alignat}{6}
K^*&=&\max\left(\left\lfloor \frac{2}{(\epsilon'_2)^2}\left(
L_0\Delta+Q+\sqrt{d}L_0^2\right)+1\right\rfloor,\left\lceil \frac{2\sqrt{d}}{(\epsilon'_2)^2}\left(
\frac{L_0\Delta+Q}{\epsilon_1}+L_0^2\right)\right\rceil\right)\nonumber
\end{alignat}
and  
\begin{alignat}{6}  
\beta^*=\frac{\log(K^*(\epsilon'_2)^2-2\sqrt{d}L_0^2)-\log(2(L_0\Delta+Q))}{\log(K^*)}.\nonumber
\end{alignat}
In order to ensure the validity of these choices for $K$ and $\beta$, we require that $0<\epsilon_2'<L_0$. The optimization problem then becomes
\begin{alignat}{6}
\min\limits_{\R,T,\psi}&\text{ }&&\R((K^*-1)\lceil (K^*)^{1-\beta^*}\rceil+T)\nonumber\\
\st&&&e^{-\R}+\frac{\R}{\psi}\leq \gamma\label{eq:opt1}\\
&&&\epsilon^2_2-6\psi\frac{Q}{T}\in(0,4e L^2_0)\label{eq:opt2}\\
&&&\R,T\in\ZZ_{>0},\quad \psi>0,\nonumber 
\end{alignat}
where \eqref{eq:opt2} ensures that $0<\epsilon_2'<L_0$. Choosing  $\R=\lceil-\ln(c\gamma)\rceil$ and $\psi=\frac{\lceil-\ln(c\gamma)\rceil}{(1-c)\gamma}$
for any $c\in(0,1)$ ensures that \eqref{eq:opt1} holds by satisfying the inequalities 
\begin{alignat}{6}
e^{-\R}\leq c\gamma\quad\text{and}\quad\frac{\R}{\psi}\leq (1-c)\gamma.\nonumber 
\end{alignat}
Choosing $T=\lceil6\phi\psi\frac{Q}{\epsilon^2_2}\rceil$ is feasible for \eqref{eq:opt2}:  
\begin{alignat}{6}
&&&4eL_0^2>\epsilon^2_2\geq\epsilon^2_2-6\psi\frac{Q}{T}=
   \epsilon^2_2-6\psi\frac{Q}{\lceil6\phi\psi\frac{Q}{\epsilon^2_2}\rceil}\geq\epsilon^2_2-6\psi\frac{Q}{6\phi\psi\frac{Q}{\epsilon^2_2}}=
   \epsilon^2_2-\frac{\epsilon^2_2}{\phi}>0,\label{longineq}
\end{alignat}
given the assumptions that $\epsilon_2<L_0$ and $\phi>1$. We have verified that the choices for $K$, $\beta$, $\R$, and $T$ ensure that the output $\bar{x}^*$ of the proposed method is an $(\epsilon_1,\epsilon_2)$-stationary point with a probability of at least $1-\gamma$. What remains is the computational complexity. 
The total number of gradient calls equals
\begin{alignat}{6}
\lceil-\ln(c\gamma)\rceil\bigg((K^*-1)\lceil (K^*)^{1-\beta^*}\rceil+\bigg\lceil6\phi\frac{\lceil-\ln(c\gamma)\rceil}{(1-c)\gamma}\frac{Q}{\epsilon^2_2}\bigg\rceil\bigg).\label{gradcalls}
\end{alignat}
From Corollary \ref{ssdcomp2}, 
$(K^*-1)\lceil (K^*)^{1-\beta^*}\rceil=O\left(\frac{1}{\epsilon_1(\epsilon_2')^4}\right)$  
and from \eqref{longineq} 
$\epsilon_2'=\sqrt{\frac{\epsilon^2_2-6\psi\frac{Q}{T}}{4e}}\geq\frac{\epsilon_2}{2}\sqrt{\frac{(1-\phi^{-1})}{e}}$, hence $(K^*-1)\lceil (K^*)^{1-\beta^*}\rceil=O\left(\frac{1}{\epsilon_1\epsilon_2^4}\right)$. The total computational complexity from \eqref{gradcalls} then equals
$\tilde{O}\left(\frac{1}{\epsilon_1\epsilon^4_2}+\frac{1}{\gamma\epsilon^2_2}\right)$.
\end{proof}

\subsection{Section \ref{num}}

\begin{customproperty}{\ref{nnl}}
	Each function $\LL(\alpha^3(H(W^3)\alpha^2(W^2v^i+b^2)+b^3),y^i)$ 
	is\\ $L_i:=2\max(\sqrt{N_2N_3}\norm{[(v^i)^T,1]}_2,\sqrt{(N_2m^2+1)})$-Lipschitz continuous.
\end{customproperty}

\begin{proof} 
	We will ultimately consider the decision variables in a vector form $\overline{w}:=[(\overline{w}^2)^T,(\overline{w}^3)^T]^T$, where $\overline{w}^l:=[W^l_1,b^l_1,W^l_2,b^l_2,...,W^l_{N_l},b^l_{N_l}]^T$ for $l=2,3$, where  $W^l_j$ is the $j^{th}$ row of $W^l$.\\
	
	The partial derivative of $\LL$ with respect to $z^3_j$ is 
	\begin{alignat}{6}
	\frac{\partial \LL}{\partial z^3_j}&=&&\alpha^3_j-y^i_j.\nonumber
	\end{alignat}	
	Given that $y^i$ is one-hot encoded, and the $\alpha^3_j$ take the form of 
	probabilities, 
	$\norm{\nabla_{z^3} \LL}_2\leq \sqrt{2}$, and $\LL$ as a function 
	of $z^3$ is $\sqrt{2}$-Lipschitz continuous. Considering $z^3=H(W^3)\alpha^{2}+b^3$ as a function of 
	$\overline{w}^3$, let 
	\begin{alignat}{6}
	&\overline{h}(\overline{w}^3)&&:=[H(W^3_1),b^3_1,H(W^3_2),b^3_2,...,H(W^3_{N_3}),b^3_{N_3}]^T\in\RR^{N_3(N_2+1)},\nonumber
	\end{alignat}  
	$\overline{\alpha}:=[(\alpha^2)^T,1]\in\RR^{N^2+1}$, $\z:=[0,...,0]\in\RR^{N^2+1}$, and the matrix 	
	$$A:=
	\begin{bmatrix}
	\overline{\alpha} &\z & \z & ... & \z\\
	\z & \overline{\alpha} & \z & ... & \z\\
	... & ... & ... & ... & ...\\
	\z & \z & \z & ... & \overline{\alpha}\\
	\end{bmatrix}
	\in\RR^{N_3\times(N_3(N_2+1))},$$
	so that $z^3(\overline{w}^3)=A\overline{h}(\overline{w}^3)$. The Lipschitz constant of $z^3(\overline{w}^3)$ is found by first
	bounding the spectral norm of $A$, which equals the square root of the largest 
	eigenvalue of 
	$$AA^T=\text{diag}(\norm{\overline{\alpha}}^2_2,\norm{\overline{\alpha}}^2_2,...])\in\RR^{N_3\times
		N_3},$$ 
	hence $\norm{A}_2=\norm{[(\alpha^2)^T,1]}_2\leq \sqrt{N_2m^2+1}$. The function 
	$\overline{h}(\overline{w}^3)$ is 1-Lipschitz continuous, and the composition of $L_i$-Lipschitz continuous 
	functions is 
	$\prod\limits_i L_i$-Lipschitz continuous 
	\cite[Claim 12.7]{shalev2014}, therefore $\LL$ is $\sqrt{2(N_2m^2+1)}$-Lipschitz continuous in 
	$\overline{w}^3$.\\
	
	Considering now $z^3$ as a function of $\alpha^2$, $z^3(\alpha^2)$ is 
	$\norm{H(W^3)}_2$-Lipschitz continuous. Given the boundedness of the hard tanh activation 
	function, $\norm{H(W^3)}_2\leq \norm{H(W^3)}_F\leq \sqrt{N_2N_3}$. The ReLU-m activation functions 
	are 1-Lipschitz continuous. As was done when computing a Lipschitz constant for $z^3(\overline{w}^3)$, to do so 
	for $z^2(\overline{w}^2)$, let $\overline{v}:=[(v^i)^T,1]$, redefine $\z:=[0,...,0]\in\RR^{N^1+1}$, and let 
	$$V:=
	\begin{bmatrix}
	\overline{v} &\z & \z & ... & \z\\
	\z & \overline{v} & \z & ... & \z\\
	... & ... & ... & ... & ...\\
	\z & \z & \z & ... & \overline{v}\\
	\end{bmatrix}\in\RR^{N_2\times(N_2(N_1+1))}.
	$$\\
	
	The Lipschitz constant for $z^2(\overline{w}^2)=V\overline{w}^2$ is then 
	$\norm{[(v^i)^T,1]}_2$. In summary, $\LL(z^3)$ is $\sqrt{2}$-Lipschitz, $z^3(\alpha^2)$ is 
	$\sqrt{N_2N_3}$-Lipschitz, $\alpha^2(z^2)$ is $1$-Lipschitz, and $z^2(\overline{w}^2)$ is 
	$||(v^i)^T,1||_2$-Lipschitz continuous, 
	hence $\LL$ is $\sqrt{2N_2N_3}||(v^i)^T,1||_2$-Lipschitz continuous in $\overline{w}^2$.\\
	
	Computing the Lipschitz constant for all decision variables,
	\begin{alignat}{6}
	&\norm{\LL(\overline{w})-\LL(\overline{w}')}_2\nonumber\\
	=&\norm{\LL(\overline{w}^2,\overline{w}^3)-L({\overline{w}^2}',\overline{w}^3)
	+L({\overline{w}^2}',\overline{w}^3)-\LL({\overline{w}^2}',{\overline{w}^3}')}_2\nonumber\\
	\leq&\norm{\LL(\overline{w}^2,\overline{w}^3)-\LL({\overline{w}^2}',\overline{w}^3)}_2
	+\norm{\LL({\overline{w}^2}',\overline{w}^3)-\LL({\overline{w}^2}',{\overline{w}^3}')}_2\nonumber\\
	\leq&\sqrt{2N_2N_3}\norm{(v^i)^T,1}_2\norm{\overline{w}^2-{\overline{w}^2}'}_2+\sqrt{2(N_2m^2+1)}\norm{\overline{w}^3-{\overline{w}^3}'}_2\nonumber\\
	\leq&\max(\sqrt{2N_2N_3}\norm{(v^i)^T,1}_2,\sqrt{2(N_2m^2+1)})(\norm{\overline{w}^2-{\overline{w}^2}'}_2+\norm{\overline{w}^3-{\overline{w}^3}'}_2)\nonumber\\
	\leq&2\max(\sqrt{N_2N_3}\norm{(v^i)^T,1}_2,\sqrt{(N_2m^2+1)})\norm{(\overline{w}^2,\overline{w}^3)-({\overline{w}^2}',{\overline{w}^3}')}_2,\nonumber
	\end{alignat} 	
	where the last inequality uses Young's inequality:
	\begin{alignat}{6}
	&&2\norm{\overline{w}^2-{\overline{w}^2}'}_2\norm{\overline{w}^3-{\overline{w}^3}'}_2
	&\leq&&\norm{\overline{w}^2-{\overline{w}^2}'}_2^2+\norm{\overline{w}^3-{\overline{w}^3}'}^2_2\nonumber\\
	\Longrightarrow&&\norm{\overline{w}^2-{\overline{w}^2}'}_2^2+2\norm{\overline{w}^2-{\overline{w}^2}'}_2\norm{\overline{w}^3-{\overline{w}^3}'}_2+\norm{\overline{w}^3-{\overline{w}^3}'}^2_2&\leq&& 2(\norm{\overline{w}^2-{\overline{w}^2}'}_2^2+\norm{\overline{w}^3-{\overline{w}^3}'}^2_2)\nonumber\\
	\Longrightarrow&&(\norm{\overline{w}^2-{\overline{w}^2}'}_2+\norm{\overline{w}^3-{\overline{w}^3}'}_2)^2&\leq&& 2(\norm{(\overline{w}^2,\overline{w}^3)-({\overline{w}^2}',{\overline{w}^3}')}_2^2)\nonumber\\
	\Longrightarrow&&\norm{\overline{w}^2-{\overline{w}^2}'}_2+\norm{\overline{w}^3-{\overline{w}^3}'}_2&\leq&&\sqrt{2} \norm{(\overline{w}^2,\overline{w}^3)-({\overline{w}^2}',{\overline{w}^3}')}_2.\nonumber
	\end{alignat}
\end{proof}

\begin{customproperty}{\ref{backprop}}
The approximate stochastic gradient $\widetilde{\nabla}\LL_i$, computed using the formulas \eqref{bp}, equals the gradient of $\LL_i$ with probability 1.
\end{customproperty}
\begin{proof}
The problematic terms within \eqref{bp} for which the chain rule does not necessarily apply are
\begin{alignat}{6}
	\frac{\partial H_{jk}}{\partial W_{jk}^3}(t)&=\ind_{\{t\geq -1\}}\ind_{\{t\leq 1\}}\quad\text{and}\quad
	\frac{\partial \alpha_j^2}{\partial z^2_j}(t)&=\ind_{\{t\geq 0\}}\ind_{\{t\leq m\}}\nonumber.\nonumber
\end{alignat}
Using PISGD,  
$\frac{\partial H_{jk}}{\partial W_{jk}^3}(t)$ 
is evaluated at $t=W^3_{jk}+z_{W^3_{jk}}$. The probability that 
$\frac{\partial H_{jk}}{\partial W_{jk}^3}(t)$ is evaluated at a point of 
non-differentiability, $|W^3_{jk}+z_{W^3_{jk}}|=1$, is zero:
\begin{alignat}{6}
	&\EE[\ind_{\{W^3_{jk}+z_{W^3_{jk}}=1\}}+\ind_{\{W^3_{jk}+z_{W^3_{jk}}=-1\}}]\nonumber\\
	=&\EE(\EE[\ind_{\{W^3_{jk}+z_{W^3_{jk}}=1\}}+\ind_{\{W^3_{jk}+z_{W^3_{jk}}=-1\}}|W^3_{jk}]).\nonumber
\end{alignat}
Defining $g(y):=\EE[\ind_{\{y+z_{W^3_{jk}}=1\}}+\ind_{\{y+z_{W^3_{jk}}=-1\}}]$, 
$$g(W^3_{jk})=\EE[\ind_{\{W^3_{jk}+z_{W^3_{jk}}=1\}}+\ind_{\{W^3_{jk}+z_{W^3_{jk}}=-1\}}|W^3_{jk}]$$
given the independence of $z_{W^3_{jk}}$ with $W^3_{jk}$. Since $z_{W^3_{jk}}$ is an absolutely continuous random variable, for any $y\in\RR$, $g(y)=0$, hence $\EE[\ind_{\{W^3_{jk}+z_{W^3_{jk}}=1\}}+\ind_{\{W^3_{jk}+z_{W^3_{jk}}=-1\}}|W^3_{jk}]=0$.\\  

The partial derivative $\frac{\partial \alpha_j^2}{\partial z^2_j}(t)$ is evaluated at $t=(W^2_{j}+z_{W^2_j})v^i+b^2_j+z_{b^2_j}$, which we rearrange as $t=z_{W^2_j}v^i+z_{b^2_j}+W^2_{j}v^i+b^2_j$ for convenience. 
The points of non-differentiability are when $z_{W^2_j}v^i+z_{b^2_j}+W^2_{j}v^i+b^2_j\in\{0,m\}$. Computing the probability of this event,
\begin{alignat}{6}
	&\EE[\ind_{\{z_{W^2_j}v^i+z_{b^2_j}=-W^2_{j}v^i-b^2_j\}}+\ind_{\{z_{W^2_j}v^i+z_{b^2_j}=m-W^2_{j}v^i-b^2_j\}}]\nonumber\\
	=&\EE(\EE[\ind_{\{z_{W^2_j}v^i+z_{b^2_j}=-W^2_{j}v^i-b^2_j\}}+\ind_{\{z_{W^2_j}v^i+z_{b^2_j}=m-W^2_{j}v^i-b^2_j\}}|W^2_{j},v^i,b^2_j]).\nonumber
\end{alignat}
Defining $g(Y^1,y^2,y^3):=\EE[\ind_{\{z_{W^2_j}y^2+z_{b^2_j}=-Y^1y^2-y^3\}}+\ind_{\{z_{W^2_j}y^2+z_{b^2_j}=m-Y^1y^2-y^3\}}]$, 
$$g(W^2_{j},v^i,b^2_j)=\EE[\ind_{\{z_{W^2_j}v^i+z_{b^2_j}=-W^2_{j}v^i-b^2_j\}}+\ind_{\{z_{W^2_j}v^i+z_{b^2_j}=m-W^2_{j}v^i-b^2_j\}}|W^2_{j},v^i,b^2_j],$$
since $(z_{W^2_j},z_{b^2_j})$ are independent of ($W^2_{j},v^i,b^2_j)$. For any $(Y^1,y^2,y^3)\in\RR^{2N_1+1}$, $z_{W^2_j}y^2+z_{b^2_j}$ is an absolutely continuous random variable, hence the probability $z_{W^2_j}y^2+z_{b^2_j}\in\{-Y^1y^2-y^3,m-Y^1y^2-y^3\}$ equals zero, and in particular 
\begin{alignat}{6}
	&\EE[\ind_{\{z_{W^2_j}v^i+z_{b^2_j}=-W^2_{j}v^i-b^2_j\}}+\ind_{\{z_{W^2_j}v^i+z_{b^2_j}=m-W^2_{j}v^i-b^2_j\}}|W^2_{j},v^i,b^2_j]=0.\nonumber
\end{alignat}

Given that the formulas \eqref{bp} evaluated at $x+z$ produce the partial derivatives of $\LL_i$ with probability 1, and $\LL_i$ is differentiable at $x+z$ with probability 1, the result follows.
\end{proof}

}

\bibliography{ll_bib}{}

\begin{thebibliography}{33}
\providecommand{\natexlab}[1]{#1}
\providecommand{\url}[1]{\texttt{#1}}
\expandafter\ifx\csname urlstyle\endcsname\relax
  \providecommand{\doi}[1]{doi: #1}\else
  \providecommand{\doi}{doi: \begingroup \urlstyle{rm}\Url}\fi

\bibitem[Aliprantis and Border(2006)]{charalambos2013}
Charalambos~D. Aliprantis and Kim~C. Border.
\newblock \emph{{Infinite Dimensional Analysis: A Hitchhiker's Guide}}.
\newblock Springer, 2006.

\bibitem[Bartle(1995)]{bartle1995}
Robert~G. Bartle.
\newblock \emph{{The Elements of Integration and Lebesgue Measure}}.
\newblock John Wiley \& Sons, 1995.

\bibitem[Bertsekas(1975)]{bertsekas1975}
Dimitri~P. Bertsekas.
\newblock Nondifferentiable optimization via approximation.
\newblock In Philip~Wolfe Michel L.~Balinski, editor, \emph{Nondifferentiable
  Optimization}, Mathematical Programming Study 3, pages 1--25. 1975.

\bibitem[Bianchi et~al.(2022)Bianchi, Hachem, and Schechtman]{bianchi2020}
Pascal Bianchi, Walid Hachem, and Sholom Schechtman.
\newblock Convergence of constant step stochastic gradient descent for
  non-smooth non-convex functions.
\newblock \emph{Set-Valued and Variational Analysis}, 2022.
\newblock Online First.

\bibitem[Bolte and Pauwels(2021)]{bolte2021}
J{\'e}r{\^o}me Bolte and Edouard Pauwels.
\newblock Conservative set valued fields, automatic differentiation, stochastic
  gradient methods and deep learning.
\newblock \emph{Mathematical Programming}, 188\penalty0 (1):\penalty0 19--51,
  2021.

\bibitem[Burke et~al.(2002)Burke, Lewis, and Overton]{burke2002}
James~V. Burke, Adrian~S. Lewis, and Michael~L. Overton.
\newblock {Approximating Subdifferentials by Random Sampling of Gradients}.
\newblock \emph{Mathematics of Operations Research}, 27\penalty0 (3):\penalty0
  567--584, 2002.

\bibitem[Burke et~al.(2020)Burke, Curtis, Lewis, Overton, and
  Sim{\~o}es]{burke2018}
James~V. Burke, Frank~E. Curtis, Adrian~S. Lewis, Michael~L. Overton, and Lucas
  E.~A. Sim{\~o}es.
\newblock {Gradient Sampling Methods for Nonsmooth Optimization}.
\newblock In Napsu Karmitsa Marko M. M\"{a}kel\"{a} Sona~Taheri Adil
  M.~Bagirov, Manlio~Gaudioso, editor, \emph{Numerical Nonsmooth Optimization},
  pages 201--225. Springer, 2020.

\bibitem[Carmon et~al.(2020)Carmon, Duchi, Hinder, and Sidford]{carmon2019}
Yair Carmon, John~C Duchi, Oliver Hinder, and Aaron Sidford.
\newblock {Lower bounds for finding stationary points I}.
\newblock \emph{Mathematical Programming}, 184\penalty0 (1):\penalty0 71--120,
  2020.

\bibitem[Chollet et~al.(2015)]{chollet2015}
Fran\c{c}ois Chollet et~al.
\newblock Keras.
\newblock https://keras.io, 2015.

\bibitem[Clarke(1990)]{clarke1990}
Frank~H. Clarke.
\newblock \emph{{Optimization and Nonsmooth Analysis}}.
\newblock SIAM, 1990.

\bibitem[Davis and Drusvyatskiy(2019)]{davis2019}
Damek Davis and Dmitriy Drusvyatskiy.
\newblock {Stochastic Model-Based Minimization of Weakly Convex Functions}.
\newblock \emph{SIAM Journal on Optimization}, 29\penalty0 (1):\penalty0
  207--239, 2019.

\bibitem[Davis et~al.(2020)Davis, Drusvyatskiy, Kakade, and Lee]{davis2018}
Damek Davis, Dmitriy Drusvyatskiy, Sham Kakade, and Jason~D. Lee.
\newblock {Stochastic Subgradient Method Converges on Tame Functions}.
\newblock \emph{Foundations of Computational Mathematics}, 20\penalty0
  (1):\penalty0 119--154, 2020.

\bibitem[Ermoliev et~al.(1995)Ermoliev, Norkin, and Wets]{ermoliev1995}
Yuri~M. Ermoliev, Vladimir~I. Norkin, and Roger J-B Wets.
\newblock {The Minimization of Semicontinuous Functions: Mollifier
  Subgradients}.
\newblock \emph{SIAM Journal on Control and Optimization}, 33\penalty0
  (1):\penalty0 149--167, 1995.

\bibitem[Federer(1996)]{federer199}
Herbert Federer.
\newblock \emph{{Geometric Measure Theory (Reprint of 1969 Edition)}}.
\newblock Springer, 1996.

\bibitem[Folland(1999)]{folland1999}
Gerald~B. Folland.
\newblock \emph{{Real Analysis: Modern Techniques and Their Applications}}.
\newblock John Wiley \& Sons, 1999.

\bibitem[Ghadimi and Lan(2013)]{ghadimi2013}
Saeed Ghadimi and Guanghui Lan.
\newblock {Stochastic First- and Zeroth-order Methods for Nonconvex Stochastic
  Programming}.
\newblock \emph{SIAM Journal on Optimization}, 23\penalty0 (4):\penalty0
  2341--2368, 2013.

\bibitem[Goldstein(1977)]{goldstein1977}
A.A. Goldstein.
\newblock {Optimization of Lipschitz continuous functions}.
\newblock \emph{Mathematical Programming}, 13\penalty0 (1):\penalty0 14--22,
  1977.

\bibitem[Keskar et~al.(2017)Keskar, Nocedal, Tang, Mudigere, and
  Smelyanskiy]{keskar2019}
Nitish~Shirish Keskar, Jorge Nocedal, Ping Tak~Peter Tang, Dheevatsa Mudigere,
  and Mikhail Smelyanskiy.
\newblock {On Large-Batch Training for Deep Learning: Generalization Gap and
  Sharp Minima}.
\newblock In \emph{5th International Conference on Learning Representations},
  2017.

\bibitem[Kornowski and Shamir(2021)]{kornowski2021}
Guy Kornowski and Ohad Shamir.
\newblock {Oracle Complexity in Nonsmooth Nonconvex Optimization}.
\newblock In M.~Ranzato, A.~Beygelzimer, Y.~Dauphin, P.S. Liang, and J.~Wortman
  Vaughan, editors, \emph{Advances in Neural Information Processing Systems},
  volume~34, pages 324--334. Curran Associates, Inc., 2021.

\bibitem[Lakshmanan and De~Farias(2008)]{lakshmanan2008}
Hariharan Lakshmanan and Daniela~Pucci De~Farias.
\newblock {Decentralized Resource Allocation in Dynamic Networks of Agents}.
\newblock \emph{SIAM Journal on Optimization}, 19\penalty0 (2):\penalty0
  911--940, 2008.

\bibitem[Liao et~al.(2004)Liao, Fang, and Nuttle]{liao2004}
Yi~Liao, Shu-Cherng Fang, and Henry~L.W. Nuttle.
\newblock A neural network model with bounded-weights for pattern
  classification.
\newblock \emph{Computers \& Operations Research}, 31\penalty0 (9):\penalty0
  1411--1426, 2004.

\bibitem[Nesterov(2004)]{nester2004}
Yurii Nesterov.
\newblock \emph{{Introductory Lectures on Convex Optimization: A Basic
  Course}}.
\newblock Springer, 2004.

\bibitem[Nesterov and Spokoiny(2017)]{nesterov2017}
Yurii Nesterov and Vladimir Spokoiny.
\newblock {Random Gradient-Free Minimization of Convex Functions}.
\newblock \emph{Foundations of Computational Mathematics}, 17:\penalty0
  527--566, 2017.

\bibitem[Pang and Tao(2018)]{pang2018}
Jong-Shi Pang and Min Tao.
\newblock {Decomposition Methods for Computing Directional Stationary Solutions
  of a Class of Nonsmooth Nonconvex Optimization Problems}.
\newblock \emph{SIAM Journal on Optimization}, 28\penalty0 (2):\penalty0
  1640--1669, 2018.

\bibitem[Rockafellar and Wets(2009)]{rockafellar2009}
R.~Tyrrell Rockafellar and Roger J-B Wets.
\newblock \emph{{Variational Analysis}}.
\newblock Springer, 2009.

\bibitem[Salimans and Kingma(2016)]{salimans2016}
Tim Salimans and Diederik~P. Kingma.
\newblock {Weight Normalization: A Simple Reparameterization to Accelerate
  Training of Deep Neural Networks}.
\newblock In D.~Lee, M.~Sugiyama, U.~Luxburg, I.~Guyon, and R.~Garnett,
  editors, \emph{Advances in Neural Information Processing Systems}, volume~29.
  Curran Associates, Inc., 2016.

\bibitem[Shalev-Shwartz and Ben-David(2014)]{shalev2014}
Shai Shalev-Shwartz and Shai Ben-David.
\newblock \emph{{Understanding Machine Learning: From Theory to Algorithms}}.
\newblock Cambridge University Press, 2014.

\bibitem[Shreve(2004)]{shreve2004}
Steven~E. Shreve.
\newblock \emph{{Stochastic Calculus for Finance II: Continuous-Time Models}}.
\newblock Springer, 2004.

\bibitem[Xu et~al.(2019)Xu, Qi, Lin, Jin, and Yang]{xu2019}
Yi~Xu, Qi~Qi, Qihang Lin, Rong Jin, and Tianbao Yang.
\newblock {Stochastic optimization for DC Functions and Non-smooth Non-convex
  Regularizers with Non-asymptotic Convergence}.
\newblock In Kamalika Chaudhuri and Ruslan Salakhutdinov, editors,
  \emph{Proceedings of the 36th International Conference on Machine Learning},
  volume~97 of \emph{Proceedings of Machine Learning Research}, pages
  6942--6951. PMLR, 2019.

\bibitem[Yang et~al.(2010)Yang, Xu, White, Schuurmans, and Yu]{yang2010}
Min Yang, Linli Xu, Martha White, Dale Schuurmans, and Yao-liang Yu.
\newblock {Relaxed Clipping: A Global Training Method for Robust Regression and
  Classification}.
\newblock In J.~Lafferty, C.~Williams, J.~Shawe-Taylor, R.~Zemel, and
  A.~Culotta, editors, \emph{Advances in Neural Information Processing
  Systems}, volume~23. Curran Associates, Inc., 2010.

\bibitem[Yousefian et~al.(2011)Yousefian, Nedi{\'c}, and
  Shanbhag]{yousefian2011}
Farzad Yousefian, Angelia Nedi{\'c}, and Uday~V. Shanbhag.
\newblock {On Stochastic Gradient and Subgradient Methods with Adaptive
  Steplength Sequences}.
\newblock \emph{arXiv preprint arXiv:1105.4549}, 2011.

\bibitem[Yousefian et~al.(2012)Yousefian, Nedi{\'c}, and
  Shanbhag]{yousefian2012}
Farzad Yousefian, Angelia Nedi{\'c}, and Uday~V Shanbhag.
\newblock On stochastic gradient and subgradient methods with adaptive
  steplength sequences.
\newblock \emph{Automatica}, 48\penalty0 (1):\penalty0 56--67, 2012.

\bibitem[Zhang et~al.(2020)Zhang, Lin, Jegelka, Sra, and Jadbabaie]{zhang2020}
Jingzhao Zhang, Hongzhou Lin, Stefanie Jegelka, Suvrit Sra, and Ali Jadbabaie.
\newblock {Complexity of Finding Stationary Points of Nonconvex Nonsmooth
  Functions}.
\newblock In Hal~Daumé III and Aarti Singh, editors, \emph{Proceedings of the
  37th International Conference on Machine Learning}, volume 119 of
  \emph{Proceedings of Machine Learning Research}, pages 11173--11182. PMLR,
  2020.

\end{thebibliography}
\bibliographystyle{plain}
\end{document}